\documentclass[a4paper]{article}
\usepackage{amsmath}
\usepackage{amsfonts}
\usepackage{amssymb}
\usepackage{amsthm}
\usepackage{graphicx}
\usepackage{subfigure}
\newtheorem{thm}{Theorem}[section]
\newtheorem{defi}[thm]{Definition}
\newtheorem{lem}[thm]{Lemma}

\newtheorem{prop}[thm]{Proposition}

\numberwithin{equation}{section}

\begin{document}

\title{\textbf{The flow of two falling balls mixes rapidly}}
\author{\textbf{P\'eter B\'alint$^{1,2}$, Andr\'as N\'emedy Varga$^{3}$}
\bigskip\\
1: MTA-BME Stochastics Research Group\\
Budapest University of Technology and Economics\\
Egry J\'ozsef u. 1, H-1111 Budapest, Hungary
\medskip \\and \medskip \\
2: Department of Stochastics, Institute of Mathematics,\\
Budapest University of Technology and Economics\\
Egry J\'ozsef u. 1, H-1111 Budapest, Hungary
\medskip \\and \medskip \\
3: MTA Alfr\'ed R\'enyi Institute of Mathematics\\
Re\'altanoda u. 13-15, 1053 Budapest, Hungary
\medskip \\
\texttt{pet@math.bme.hu,nemedy@math.bme.hu}\\
}
\date{\today}
\maketitle

\begin{abstract}
In this paper we study the system of two falling balls in continuous time. We modell the system by a suspension flow over a two dimensional, hyperbolic base map. By detailed analysis of the geometry of the system we identify special periodic points and show that the ratio of certain periods in continuous time is Diophantine for almost every value of the mass parameter in an interval. Using results of Melbourne (\cite{M}) and our previous achievements \cite{BBNV} we conclude that for these values of the parameter the flow mixes faster than any polynomial. Even though the calculations are presented for the specific physical system, the method is quite general and can be applied to other suspension flows, too.
\end{abstract}

\section*{Introduction}

One of the main motivations for studying the statistical properties of hyperbolic dynamical systems is related to applications in physics. Yet, models with direct physics relevance are typically complicated, and thus results concerning such systems are quite limited. A remarkable exception is the class of dispersing billiards, see \cite{CM} for a detailed description of their theory.

The system of falling balls investigated in the present paper cannot be regarded as a small perturbation of a dispersing billiard, especially as far as the dynamics in continuous time is concerned. This model introduced by Wojtkowski in \cite{W} describes the motion of $n$ point masses along a vertical half-line under the action of gravity, which collide elastically with each other and the floor. This system can be considered both in discrete and in continuous time. The first results were on hyperbolicity, i.e. non-vanishing of all relevant Lyapunov exponents. In \cite{W} Wojtkowski proved hyperbolicity in case the masses of the particles are strictly decreasing up the line. Later Sim\'anyi weakened this assumption, in \cite{S} he showed hyperbolicity, when the masses of the particles decrease (but not necessarily strictly) up the line and there are at least two different masses. Aiming at finer chaotic properties, in \cite{LW} Wojtkowski and Liverani developed a general method to show ergodicity for Hamiltonian systems and this could be used to show ergodicity of the system of \emph{two} falling balls, in case the lower ball is heavier. Earlier Chernov showed a similar result in \cite{C1}. However, for three or more particles ergodicity is still an open question. In the ergodic regime of the two falling balls a detailed geometric description of the system and a quantitative analysis of the \textit{discrete time} map lead to the verification of fine statistical properties in \cite{BBNV}, in particular polynomial decay of correlations and the central limit theorem for H\"older observables.

In this paper we investigate the system of two falling balls in \textit{continuous time}. Studying statistical properties of hyperbolic flows is a technically involved task which is mostly related to the lack of hyperbolic behavior in the flow direction. In the past two decades there has been substantial activity in this field, here we summarize some results that are closely related to our work. An essential breakthrough was initiated by Chernov in \cite{C2}, where, using Markov approximations, he obtained a stretched exponential bound on time correlations for $3$-dimensional Anosov flows that verify the so called `\textit{uniform nonintegrability of foliations}' condition (UNI for short). Dolgopyat developed Chernov's result in two different directions. On the one hand, in \cite{D1} he showed that Anosov flows satisfying the UNI condition and a high degree of regularity are exponentially mixing. In his later work he studied the more general class of suspension flows over subshifts of finite type and in \cite{D} he proved that such flows are typically rapid mixing, meaning that time correlations for sufficiently regular observables decay faster than any polynomial. Here the typicality condition is related to the presence of two periodic orbits such that the ratio of their periods satisfies a Diophantine condition. This condition plays an important role in our paper, thus we introduce the abbreviation DPO (Diophantine periodic orbits) for later reference. Dolgopyat's result on rapid mixing was extended by Melbourne (\cite{M}) to suspensions over Gibbs-Markov maps and also to suspensions over hyperbolic maps that can be modelled by a Young-tower (\cite{Y1},\cite{Y2}) with exponential tails.

The DPO condition is a much weaker form of non-integrability in the flow direction than the UNI condition. Yet, in most applications, the verification of either UNI or DPO is based on the invariance of a canonical contact form or a perturbation thereof (see \cite{L} or \cite{BL} for example).\footnote{In that respect, dispersing billiard flows may be regarded as singular geodesic flows.} Other than that, we are only aware of results that prove DPO -- and hence rapid mixing -- for a class of flows that is residual in an appropriate topology, and not for specific examples. Note that for Hamiltonian flows on cotangent bundles the canonical contact form associated to the
symplectic form is preserved only if the Hamiltonian is a homogeneous function of the momenta (see \cite{KH}, section 5.6). For the flow of two falling balls by the presence of a non-infinitesimal external field the Hamiltonian cannot be regarded as a small perturbation of a homogeneous function, hence we seek for alternative methods.

In the present paper we prove DPO for the system of two falling balls, for almost every value of the mass ratio within a large interval (in the ergodic regime). To conclude that the system mixes rapidly in continuous time, we rely on \cite{M}. This requires some additional work, as the periodic points originally constructed are a macroscopic distance apart, while for \cite{M} it is essential that they are present on the base of the same Young tower. To obtain periodic points that can be realized on the base of the same Young tower, we apply a shadowing type argument, the details of which require most of the technical work in this paper. Similar ideas have already appeared in the literature, see in particular the notion of ``periodic points with good asymptotics'' in \cite{FMT}. Nonetheless, in \cite{FMT} good asymptotics is used to conclude stability of mixing and rapid mixing in the $C^r$ topology, while here we implement a shadowing type argument for the specific system of falling balls.

As a consequence of our analysis we conclude that the system of two falling balls mixes rapidly in continuous time, for a set of mass ratios that has positive Lebesgue measure (cf. Theorem \ref{aediophantine}). It is worth pointing out that our analysis applies for almost every value of the mass ratio, as long as the mass of the lower ball is at least one and a half times larger than the mass of the upper ball. The only reason why we have to restrict to a smaller set of mass ratios is that in \cite{BBNV} the presence of a Young tower is established only for a smaller, yet open set of mass ratios. See our remarks after Formula \eqref{ratioextended} for further discussion.

The rest of this paper is organized as follows. In section~\ref{s:setup} we summarize the necessary prerequisites concerning the system of two falling falling balls, mostly from \cite{BBNV}. In section~\ref{s:results}
we discuss suspension flows and state our main results. Section~\ref{calculations} contains the core argument of the paper, the construction of the periodic points satisfying DPO, along with the shadowing type argument, for mass ratios $m\in(\frac23,\frac34)$. Finally, section~\ref{s:extension} discusses the extension to other values of the mass ratio.

\section{Setup and notations \label{s:setup}}

In this section we introduce the system and recall the necessary notations and results from our earlier paper \cite{BBNV}. The exposition is self contained, for further details about the dynamics and its properties we refer to our previous work.

The system of two falling balls, introduced by Wojtkowski in \cite{W}, describes the motion of two point particles of masses $m_1$ and $m_2$ that move along the vertical half-line, subject to constant gravitational force, and collide elastically with each other and the floor. We consider the case when the lower ball is heavier (i.e. $m_1>m_2$), which corresponds to ergodic and hyperbolic dynamics (as shown in \cite{LW} and \cite{W}). As the action of ball to ball collisions depends only on the ratio of the two masses we rescale these masses such that $m_1+m_2=1$. We introduce our mass parameter $m$ and from now on we use the notation that $m_1 = m$, $m_2=1-m$, where $m \in (1/2,1)$ since we are in the ergodic case.

\begin{figure}[h!]
\centering
\includegraphics[scale=0.6]{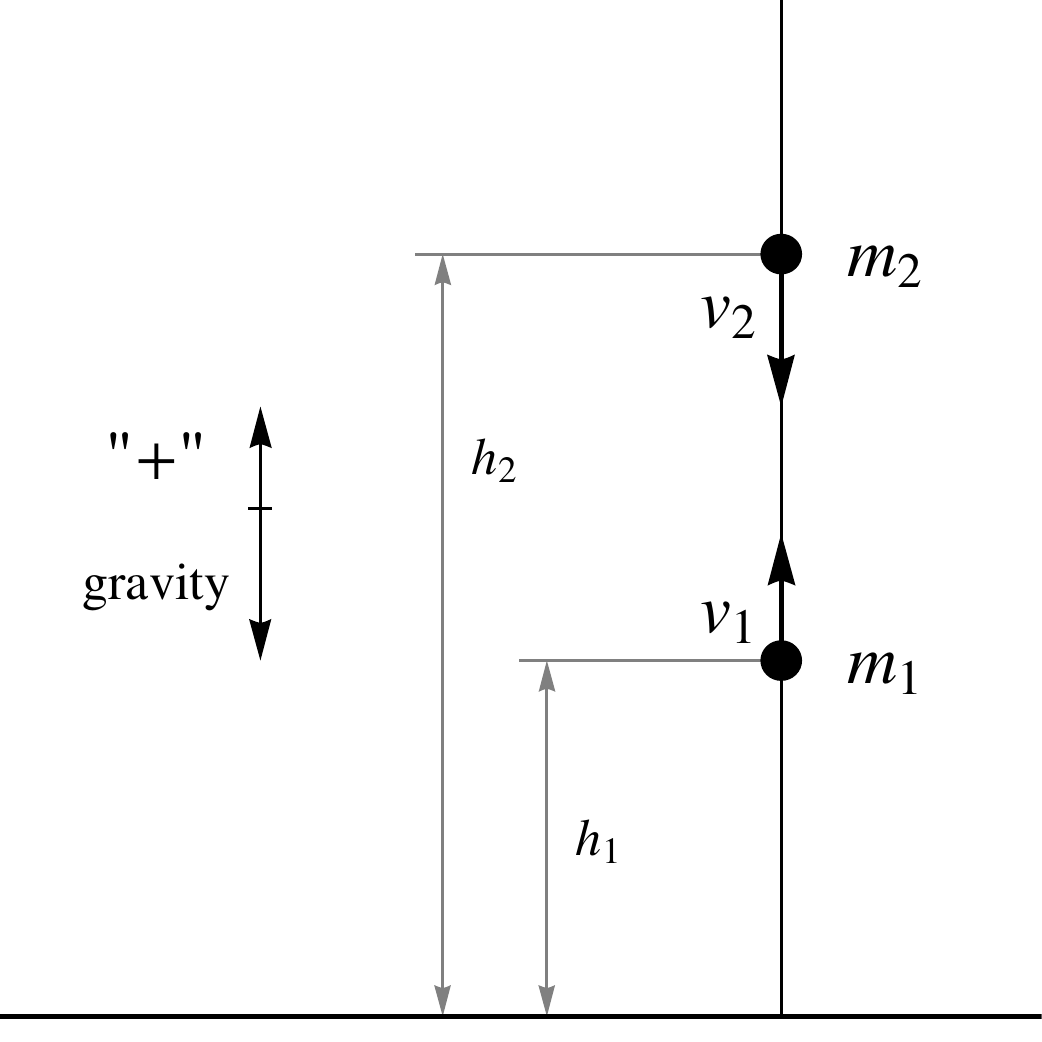}
\caption{The system of two falling balls}
\label{sima}
\end{figure}

We neglect air resistance and assume all collisions to be totally elastic, therefore the flow preserves the total energy of the system, which we set to be $1/2$ for practical reasons. We discretize time by considering the outgoing Poincar\'e section corresponding to moments when the lower ball hits the floor and the next collision will occur between the two balls (and not between the lower one and the floor). Based on the work of Wojtkowski we use the coordinates
\[ h = \frac12 m v_1^2 \quad \text{and} \quad z = v_2 - v_1, \]
to describe the system, where $v_i$ is the signed velocity of the $i$'th ball. This means that $h$ is the energy of the lower ball (since it is on the floor it only has kinetic energy) and $z$ is the difference of the velocities. Note that these coordinates are invariant between collisions. The phase space is then

\begin{equation*}
\begin{split}
\mathcal{M}_1:=\Bigl\{(h, z)\in\mathbb{R}^2 |\, 0 < h < 1/2, \, 1/2-h > \frac{1}{2} (1-m) \biggl(z + \sqrt{\frac{2h}{m}}\biggr)^2,\\ \, m(1-m) z \biggl(2 \sqrt{\frac{2 h}{m}}-z\biggr)-2 h + m < 0 \Bigr\},
\end{split}
\end{equation*}
where the conditions arise as follows.

\begin{enumerate}
	\item The first condition says that the energy of the lower ball is positive, but not greater than the total energy of the system, which we set to $1/2$ previously.
	\item The second condition is the inequality that implies that the upper ball has positive height.
	\item Finally, the third condition is to ensure that the two balls will collide before the lower ball returns to the floor.
\end{enumerate}
Recall that our Poincar\'e section $\mathcal{M}_1$ corresponds to situations in which
\begin{enumerate}
    \item the lower ball is on the floor and
    \item it will collide with the upper one before returning to the floor.
\end{enumerate}
Hence, starting from a configuration $(h,z) \in \mathcal{M}_1$, first the two balls collide, and then the lower one will hit the floor several times before getting back to $\mathcal{M}_1$. Let us denote by $R(h,z)$ the number of bumps of the lower ball on the floor before returning to $\mathcal{M}_1$, starting from the configuration $(h,z)$. Then for any $n \in \mathbb{N}$ we introduce
\[ R_n := \{ (h,z) \in \mathcal{M}_1 | R(h,z) = n \}. \]
It is shown in \cite{BBNV} that none of these sets are empty. They are of course disjoint, moreover, even the closures of any two of them are disjoint provided their indices differ by more than $1$. This way, in fact,
the domains of continuity for the dynamics are identified, which provide a partition of the phase space $\mathcal{M}_1$.

\begin{figure}[h!]
\centering
\includegraphics[scale=0.4]{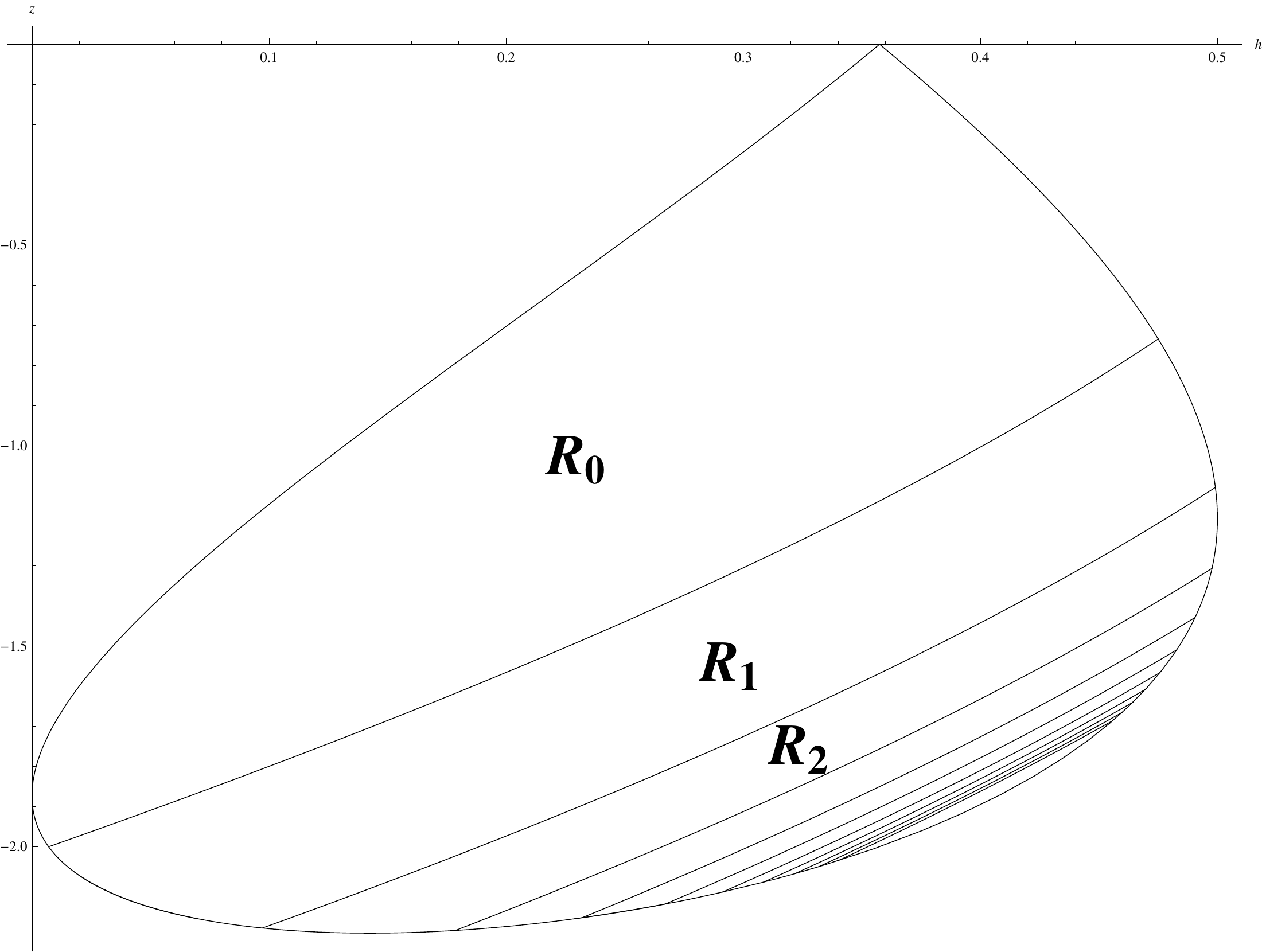}
\caption{The phase space of the dynamics}
\label{phasespace}
\end{figure}

Indeed, let us denote the dynamics of the system by $T : \mathcal{M}_1 \to \mathcal{M}_1$. Then using the classical Newtonian laws of mechanics one can calculate that

\begin{equation}\label{TonRn}
T(h,z) = (m F , -(2n+2) \sqrt{2F} - z) \quad \text{for } (h,z) \in R_n,
\end{equation}
where

\begin{equation}\label{Fandalpha}
F = 1 - h/m + \alpha z^2, \quad \text{and } \, \alpha = 1-3m+2m^2.
\end{equation}
The latter quantities often occur in the formulas, that is why we introduced extra notation on them. It can be seen from $\eqref{TonRn}$ that $T$ is continuous, moreover $C^2$ on each $R_n$, but the crucial dependence of $T$ on $n$ shows that $T$ is discontinuous on $\cup_{n=0}^{\infty} \partial R_n$. The discontinuities occur at the curves

\begin{equation}\label{discontinuity}
r_n := \partial R_n \cap \partial R_{n+1}.
\end{equation}
These  correspond to configurations starting from which the two balls collide, then the lower one hits the floor $n$ times and finally \textit{they land on the ground at the same time, on top of each other}. Hence it is not clear whether the lower one reached the floor before their collision, or if it was the other way around. The two possible cases correspond to two different limits, one where the initial point $(h,z) \in r_n$ is approached from inside $R_{n+1}$ and the other where it is approached from inside $R_n$.
\\As we pointed out $T$ maps each $R_n$ diffeomorphically onto its image. The jacobian of the dynamics is
\begin{equation}\label{jacobian}
DT|_{R_n}(h,z) = \left( \begin{array}{cc}
                         -1                                                  & 2 m \alpha z \\
                         \frac{\sqrt{2}(n+1)}{m \sqrt{F}} & -1-\frac{(2n+2)\sqrt{2} \alpha z}{\sqrt{F}}
                       \end{array} \right).
\end{equation}
An important consequence of this formula is that $det(DT(h,z)) = 1$ and hence the normalized Lebesgue measure on $\mathcal{M}_1$ is an absolutely continuous invariant probability measure of the dynamics, which, by ergodicity, is unique. It can also be derived that $DT(h,z)$ is a hyperbolic matrix at every point and that it contains a rotation by $180$ degrees.
\\Figure \ref{imageofstripes} demonstrates how $T$ maps the set $R_n$ onto its image. For further details about the regularity properties of $T$ \cite{BBNV}, section 3 is referred. The important fact that we will use in this paper is that $T$ is uniformly hyperbolic (proved in \cite{BBNV} subsection 3.3), in the sense that there exists a forward invariant unstable, and a backward invariant stable cone field ($C_x^u$ and $C_x^s$ respectively) and these cone fields are uniformly transversal to each other. Curves $\gamma$ such that the tangent line $T_x\gamma$ lies in the unstable cone $C_x^u$ for every $x \in \gamma$, are referred to as unstable curves. Stable curves are defined in an analogous way. Again in \cite{BBNV} it is shown that stable curves are increasing, while unstable curves are decreasing in the $(h,z)$ coordinates.

\begin{figure}[h!]
\centering
\includegraphics[scale=0.6]{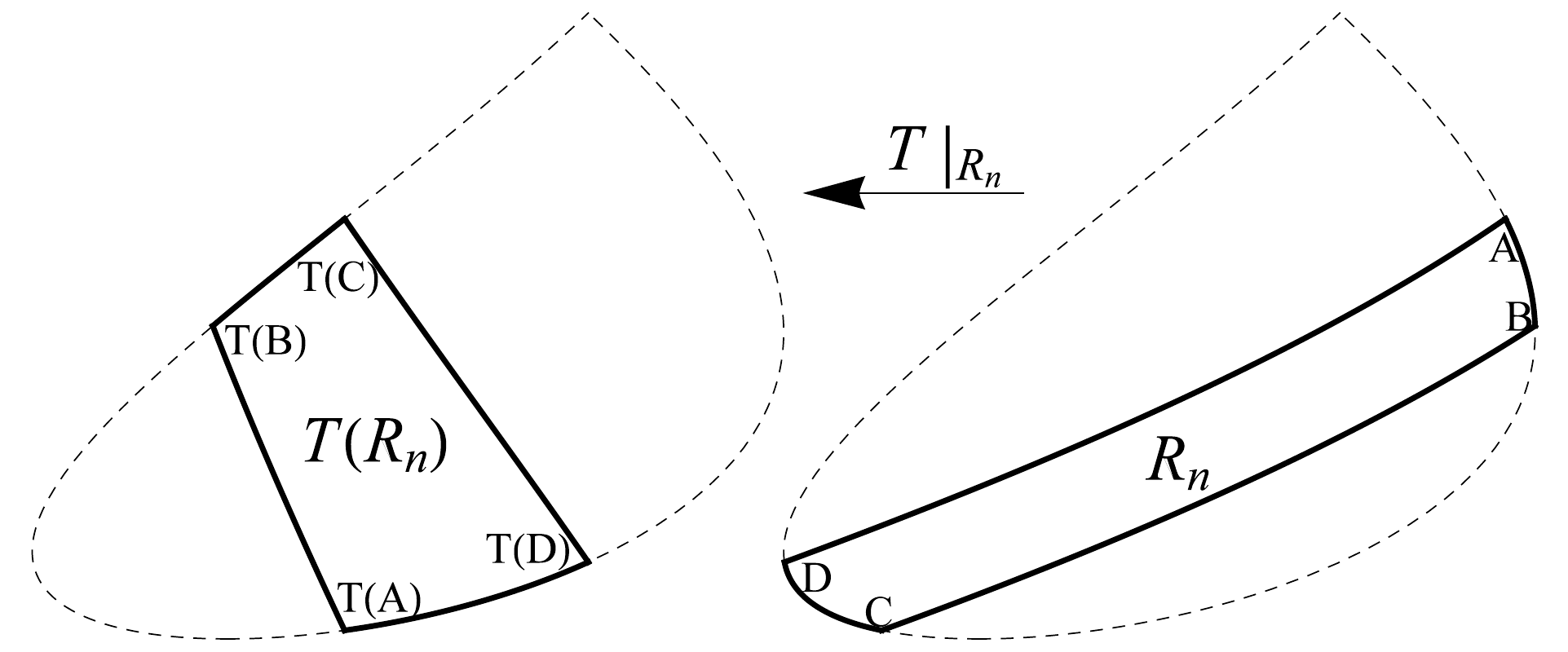}
\caption{The action of the dynamics on $R_n$}
\label{imageofstripes}
\end{figure}

Concerning terminology it is important to distinguish stable manifolds as special stable curves. The stable manifold of a point $x\in\mathcal{M}_1$ is defined as the curve $W^s(x)$ such that for $y\in W^s(x)$ we have $d(T^nx,T^ny)\to 0$ exponentially as $n\to\infty$. Equivalently, $W^s$ is a stable manifold if $T^nW^s$ is a smooth stable curve for any $n\ge 0$. By the general theory of hyperbolic systems with singularities, a unique stable manifold of positive length passes through almost every $x\in\mathcal{M}_1$ (see eg.~\cite{CM} and references therein). Unstable manifolds are special unstable curves with analogous properties.

In our previous paper we proved the following statements (cf.~\cite{BBNV}, subsection 1.2 and section 5.)

\begin{thm}\label{youngtower}
There exists an open interval $I \subseteq (1/2,1)$ such that for any mass ratio $m \in I$ the discrete time map $T : \mathcal{M}_1 \to \mathcal{M}_1$ can be modelled by a Young-tower with exponential tails.
\end{thm}

\begin{thm}\label{complexity}
If the system has subexponential complexity for some $m \in (1/2,1)$, then the discrete time map $T : \mathcal{M}_1 \to \mathcal{M}_1$ can be modelled by a Young-tower with exponential tails for this $m \in (1/2,1)$.
\end{thm}

We also recall from our previous work that there is an involution for the map $T$, i.e. there is a smooth map $I : \mathcal{M}_1 \to \mathcal{M}_1$ such that

\begin{equation}\label{invdefi}
T^{-1} = I \circ T \circ I,
\end{equation}
on every smoothness component of $T^{-1}$. This corresponds to the natural time reflection in the continuous time system and its action in the coordinates $(h,z)$ is given by

\begin{equation}\label{involformula}
I(h,z) = (m(1-h/m+\alpha z^2),z) = (m F,z).
\end{equation}
It can be deduced that $I$ maps each $R_n$ to $T(R_n)$, and it follows from \eqref{invdefi} that $I$ maps stable curves into unstable curves and vice versa, moreover it maps stable manifolds into unstable manifolds and vice versa. Also, the curves $I(r_n)$ are the singularities of the inverse dynamics.
Finally, let us recall from \cite{BBNV} the notations for the corner points of the sets $R_n$ and $T(R_n)$. These are the intersections of the singularities $r_n$, or the inverse singularities $I(r_n)$ with the boundary of the phase space. All the curves $r_n$ are stable curves, in particular they are increasing and hence it makes sense to talk about their left and right endpoints. They are given, respectively, by the formulas

\begin{equation}\label{pointsXBx}
\begin{split}
Bx_n(m) &= (Bxh_n(m), Bxz_n(m)) \\ &= \Bigl(\frac{m(-2m(n+2)+2n+3)^2}{2(1-m)n(n+2)+2}, -\frac{n+2}{\sqrt{1-(m-1)n(n+2)}}\Bigr), \\
X_n(m) &= (Xh_n(m), Xz_n(m)) \\ &= \Bigl(\frac{m(3+2n-2m(n+1))^2}{2(n+2)^2-2m(n+1)(n+3)}, -\frac{n+1}{\sqrt{(n+2)^2-m(n+1)(n+3)}}\Bigr).
\end{split}
\end{equation}

\begin{figure}[!h]
\centering
\includegraphics[scale=0.7]{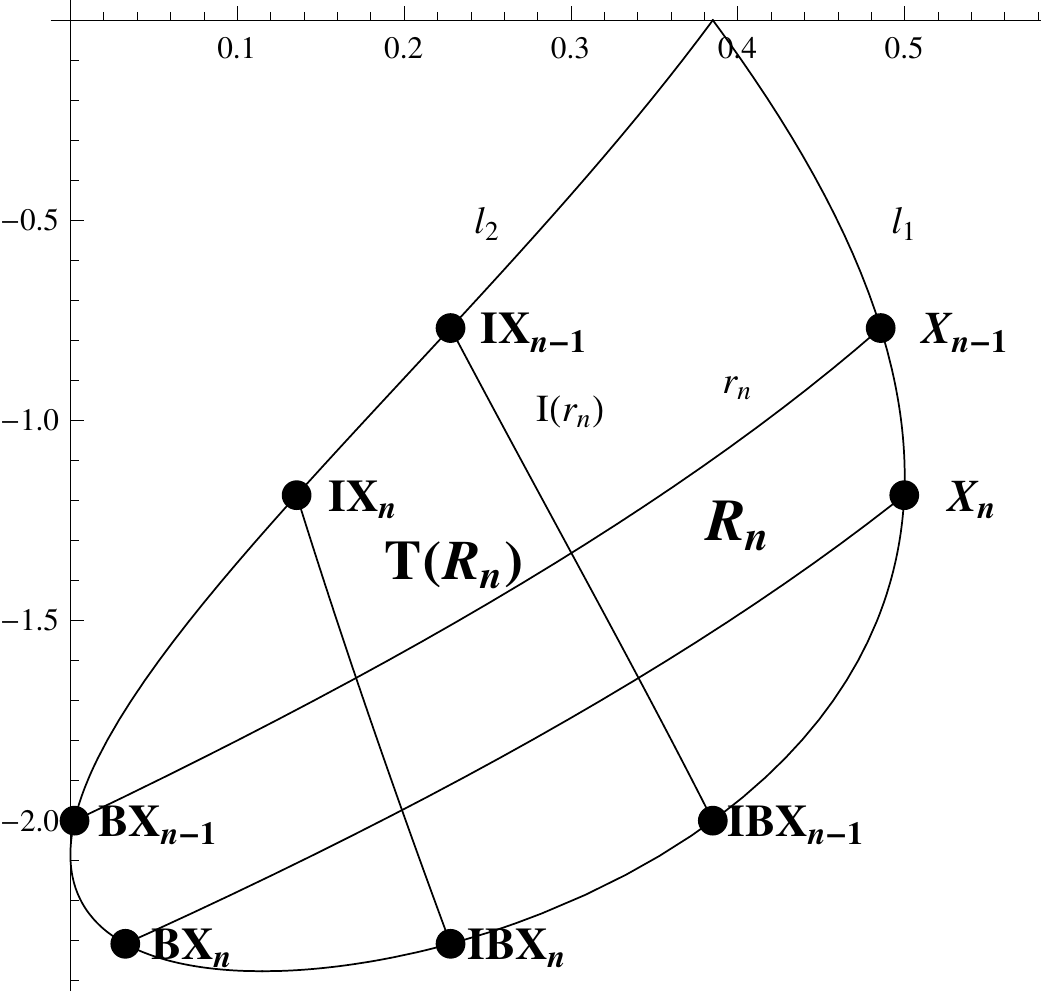}
\caption{The corners of the sets $R_n$ and $T(R_n)$}
\label{cornerpoints}
\end{figure}

The curves $I(r_n)$ are decreasing and so their left and right (or in this case rather top and bottom) endpoints are well defined, too. They are given, respectively, by the formulas

\begin{equation}\label{pointsIxIbx}
\begin{split}
Ix_n(m) &= (Ixh_n(m), Ixz_n(m)) \\ &= \Bigl(\frac{m}{2(n+2)^2-2m(n+1)(n+3)}, -\frac{n+1}{\sqrt{(n+2)^2-m(n+1)(n+3)}}\Bigr), \\
Ibx_n(m) &= (Ibxh_n(m), Ibxz_n(m)) \\ &= \Bigl(\frac{m}{2(1-m)n(n+2)+2}, -\frac{n+2}{\sqrt{1-(m-1)n(n+2)}}\Bigr).
\end{split}
\end{equation}

\section{Statement of results \label{s:results}}

To state our results we model the continuous time dynamics of the system of two falling balls by a suspension flow over the uniformly hyperbolic base map $T$, discussed in the previous section. Using the classical laws of Newtonian mechanics one can calculate how much time is needed for the flow to return to the Poincar\'e section $\mathcal{M}_1$ starting from the point $(h,z) \in \mathcal{M}_1$. It is given as

\begin{gather}\label{rooffunction}
\nonumber \tau : \mathcal{M}_1 \to \mathbb{R}^+, \\ \tau(h,z) = (2n+1)\sqrt{2F} + \sqrt{\frac{2h}{m}} - 2(m-1)z  \quad \text{for } (h,z) \in R_n.
\end{gather}
It can be shown that $\tau$ is piecewise $C^2$ with the same discontinuities as the discrete time dynamics $T$ (see \cite{BBNV}, subsection 3.7 for details). The flow is then isomorphic to the following suspension.

The phase space for the suspension flow is defined to be

\begin{equation}\label{suspensionphasespace}
\tilde{\mathcal{M}} = \{ (x,t) | x \in \mathcal{M}_1, 0 \leq t < \tau(x) \},
\end{equation}
and, after setting the equivalence relation $(x,\tau(x)) \sim (T(x),0)$, the continuous time action of the dynamics is given by

\begin{equation}\label{suspensionflow}
\Phi^t : \tilde{\mathcal{M}} \to \tilde{\mathcal{M}}, \quad \Phi^t(x,s) = (x,s+t)/\sim.
\end{equation}
Finally, the normalized Lebesgue measure on $\tilde{\mathcal{M}}$ is an ergodic invariant measure for this suspension flow.

We define a class of observables in the following way. For a function $v : \tilde{\mathcal{M}} \to \mathbb{R}$ we set its $\eta$-norm to be
\[ \|v\|_{\eta} := \|v\|_{\infty} + \sup\limits_{x \neq y}|v(x,u)-v(y,u)|/d(x,y)^{\eta}. \]
Then for $m \in\mathbb{N}$ and $\eta > 0$, let $C^{m,\eta}(\tilde{\mathcal{M}})$ the collection of $v : \tilde{\mathcal{M}} \to \mathbb{R}$ such that $\|v\|_{m,\eta} := \|v\|_{\eta} + \|\partial_t v\|_{\eta} + \dots + \|\partial_t^m v\|_{\eta} < \infty$, where $\partial_t$ denotes differentiation in the flow direction.

\begin{defi}
The suspension flow $\Phi^t$ is \emph{rapid mixing} if for any $n \geq 1$ there exist $m \geq 1$, $C \geq 0$ such that
\[ \biggl|\int\limits_{\tilde{\mathcal{M}}} v \cdot w \circ \Phi^t \, d Leb - \int\limits_{\tilde{\mathcal{M}}} v \, d Leb \cdot \int\limits_{\tilde{\mathcal{M}}} w \, d Leb\biggr| \leq C \|v\|_{m,\eta} \|w\|_{m,\eta} t^{-n}, \]
for every $v, w \in C^{m,\eta}(\tilde{\mathcal{M}})$ and $t > 0$.
\end{defi}

Here we state our results.

\begin{thm}\label{aediophantine}
The continuous time dynamics of the system of two falling balls is rapid mixing for almost every $m \in I$, where $I \subset (1/2,1)$ is the interval from Theorem \ref{youngtower}.
\end{thm}

\begin{thm}\label{assumingcomplexity}
Assuming subexponential complexity for the base map $T$ (so that Theorem \ref{complexity} can be applied), the continuous time dynamics of the system of two falling balls is rapid mixing for almost every $m \in [2/3,1)$.
\end{thm}

\section{Periodic orbits}\label{calculations}

\subsection{Outline of our strategy}
As already mentioned in the introduction, the core of our argument concerns the existence of periodic orbits with sufficiently diverse periods in the following sense.
If a point $x\in\mathcal{M}_1$ is periodic for $T$ with discrete period $k$
(ie.~$T^kx=x$) then it is also periodic for the suspension flow with flow period $\tau_k(x)=\tau(x)+\tau(Tx)+...+\tau(T^{k-1}x)$. Given two periodic points $x$ and $y$ with the same discrete period $k$, their (flow)
period ratio is defined as $\frac{\tau_k(x)}{\tau_k(y)}$. Furthermore, recall (eg.~from \cite{Bu}) that a number $\omega\in\mathbb{R}$ is Diophantine if it is badly approximable by rationals, ie.~if there exist $K>0$ and $\beta>1$ such that for any $p,q\in\mathbb{Z}$, $q\ne0$  we have $|\omega q-p|\geq K|q|^{-\beta}$.  The set of Diophantine numbers is of full Lebesgue measure. We will say that two periodic points have a Diophantine
period ratio if $\frac{\tau_k(x)}{\tau_k(y)}$ is Diophantine.

In his work \cite{D}, Dolgopyat showed that mixing suspension flows over subshifts of finite type are rapid mixing if there exist two periodic orbits for the flow with Diophantine period ratio. Later in \cite{M} Melbourne extended the approach to a large class of nonuniformly hyperbolic flows using operator renewal theory (see also \cite{MT}). These flows are the continuous time analogue of the nonuniformly hyperbolic maps studied in \cite{Y1}. He showed that if the system is not rapid mixing, then there must be some resonance in the roof function. In particular if there exist four periodic points with periods satisfying a Diophantine type relation, then the suspension flow is rapid mixing (see \cite{M} Theorem 2.6). The proof of this result contains "double inducing", first reducing from the flow to a nonuniformly hyperbolic diffeomorphism and then from this to a uniformly hyperbolic one. However one could do this in one go, immediately reducing to a uniformly hyperbolic base map at the cost of having a larger roof function. The advantage is that by having only one inducing, there is only one complex variable to appear in the proofs (\cite{M1}). Therefore, if one can find two periodic orbits on the base of the Young tower with Diophantine period ratio, then one can conclude rapid mixing for the flow (cf.~\cite{M}, Remark 2.7).

We recall the formulas for the induced dynamics $T : \mathcal{M}_1 \to \mathcal{M}_1$ and for the roof function $\tau : \mathcal{M}_1 \to \mathbb{R}$. They are given by

\begin{equation}\label{dynamics}
\begin{split}
T(h,z) = (m F , -(2n+2) \sqrt{2F} - z) \quad \text{for } (h,z) \in R_n \\
\tau(h,z) = (2n+1)\sqrt{2F} + \sqrt{\frac{2h}{m}} - 2(m-1)z \quad \text{for } (h,z) \in R_n
\end{split}
\end{equation}
where $F = 1 - h/m + \alpha z^2$ and $\alpha = 1-3m+2m^2$. We emphasize their dependence on the parameter $m$.

When searching for periodic points the easiest attempt is to look for fixed points. Simply by solving the equation $T(h,z) = (h,z)$ we get the candidates

\begin{equation}\label{fixedpoints}
F_n(m) = \biggl(\frac{m}{2(1-\alpha(n+1)^2)}, \frac{-(n+1)}{\sqrt{1-\alpha(n+1)^2}}\biggr).
\end{equation}
Even though these are solutions for the fixed point equation, not all of them are physical solutions, i.e. some of the $F_n(m)$'s may lie outside of the phase space and hence do not correspond to a valid physical configuration. However one can check that $F_0(m)$ and $F_1(m)$ are always physical solutions. Indeed both relations $F_0(m) \in R_0$ and $F_1(m) \in R_1$ are satisfied for every $m \in (1/2, 1)$. We calculate the periods of these points in continuous time

\begin{equation}\label{periods}
\tau(F_0(m), m) = \frac{2m}{\sqrt{3m-2m^2}} \quad \tau(F_1(m), m) = \frac{2m}{\sqrt{-3/4+3m-2m^2}},
\end{equation}
and consider the ratio of the two periods

\begin{equation}\label{periodratio}
\frac{\tau(F_1(m), m)}{\tau(F_0(m), m)} = \frac{\sqrt{3m-2m^2}}{\sqrt{-3/4+3m-2m^2}}.
\end{equation}
This function is $C^1$ on $[1/2, 1]$, strictly decreasing on $[1/2, 3/4]$ and strictly increasing on $[3/4, 1]$. As a consequence the ratio \eqref{periodratio} is Diophantine for almost every $m \in (1/2, 1)$. Yet, to conclude that the continuous time system mixes rapidly we can not just directly apply the results of Melbourne, because the points $F_0$ and $F_1$ have a macroscopic distance in the phase space and so they might not be represented on the base of the same tower as required in \cite{M}.

Our strategy will be as follows. With the help of the natural partition of the phase space we switch to a symbolic space. Using the geometric properties of the map and the partition elements, for certain values of the mass parameter we guarantee the existence of a subsystem, which is a two-sided full shift on the two symbols $0$ and $1$. We construct a sequence of periodic points $\{P_n(m)\}$ that accumulates on $F_0(m)$, but in a way that the trajectory of $P_n(m)$ spends more and more time in the vicinity of $F_1(m)$ as $n$ increases. We also show that all points $F_0(m)$ and $\{P_n(m)\}$ have sufficiently long stable and unstable manifolds, hence a tower can be built such that for large enough $n$ both $F_0(m)$ and $P_n(m)$ are on the base of it. In terms of \cite{M}, for rapid mixing it is then enough to show that the ratio of the continuous periods of $P_n(m)$ and $F_0(m)$ is Diophantine.

\subsection{The full shift as a symbolic subsystem}
For a point $P \in \mathcal{M}_1$ we adjust the two-sided infinite sequence $\underline{x}=\{x_i\}$ as its symbolic representation, where $x_i = n$ iff $T^i(P) \in R_n$. We denote the natural projection from the symbolic space to $\mathcal{M}_1$ by $\pi$, set the past and future separation time for two sequences $\underline{x}$ and $\underline{y}$ as

\begin{gather*}
s_-(\underline{x}, \underline{y}) = \min\{ |k| : k < 0, x_k \neq y_k \}, \\
s_+(\underline{x}, \underline{y}) = \min\{ k : k \geq 0, x_k \neq y_k \},
\end{gather*}
and also the separation time as
\[ s(\underline{x}, \underline{y}) = \min\{ |k| : x_k \neq y_k \} = \min\{s_-(x,y),s_+(x,y)\}. \]
We define the symbolic distance as
\[ d_{sym}(\underline{x}, \underline{y}) = \theta^{s(\underline{x}, \underline{y})}, \]
where $\theta \in (0,1)$ is chosen in such a way that the distance on $\mathcal{M}_1$ and in the symbolic space are related as

\begin{equation}\label{metricsequiv}
d(\pi(\underline{x}), \pi(\underline{y})) \leq C \cdot d_{sym}(\underline{x}, \underline{y}),
\end{equation}
for some constant $C > 0$. By uniform hyperbolicity of the discrete time system (proved in \cite{BBNV}) there exists such a $\theta \in (0,1)$.

We now show the existence of the two-sided infinite, full shift subsystem.

\begin{lem}\label{subsystem}
For every $m \in [2/3, 3/4]$ the projection of any sequence $\underline{x} \in \{0,1\}^{\mathbb{Z}}$ is realised as a physical configuration, i.e. as a point in $\mathcal{M}_1$. Any such point has local stable and unstable manifolds that fully cross the phase space.
\end{lem}

\begin{proof}
We will call a region quadrangular if it is diffeomorphic to a square and we will refer to such a region as a (curvilinear) rectangle if it is bounded by two stable (increasing) and two unstable (decreasing) curves. These will be referred to as the stable and the unstable sides of the rectangle, respectively.

Consider the geometry of the sets $R_i \cap T(R_j)$ for $i,j \in \{0,1\}$. While $R_0 \cap T(R_0)$ is quadrangular for every $m \in [1/2,1)$ this does not necessarily hold for the other three sets. They can be triangular, quadrangular or pentagonal (in the previous topological sense) depending on the value of the mass parameter $m$. The different possible cases are shown on Figure \ref{RicapRj}.
\begin{figure}
\centering
\subfigure[$1/2 < m \leq 7/12$]{\includegraphics[width=0.3\textwidth]{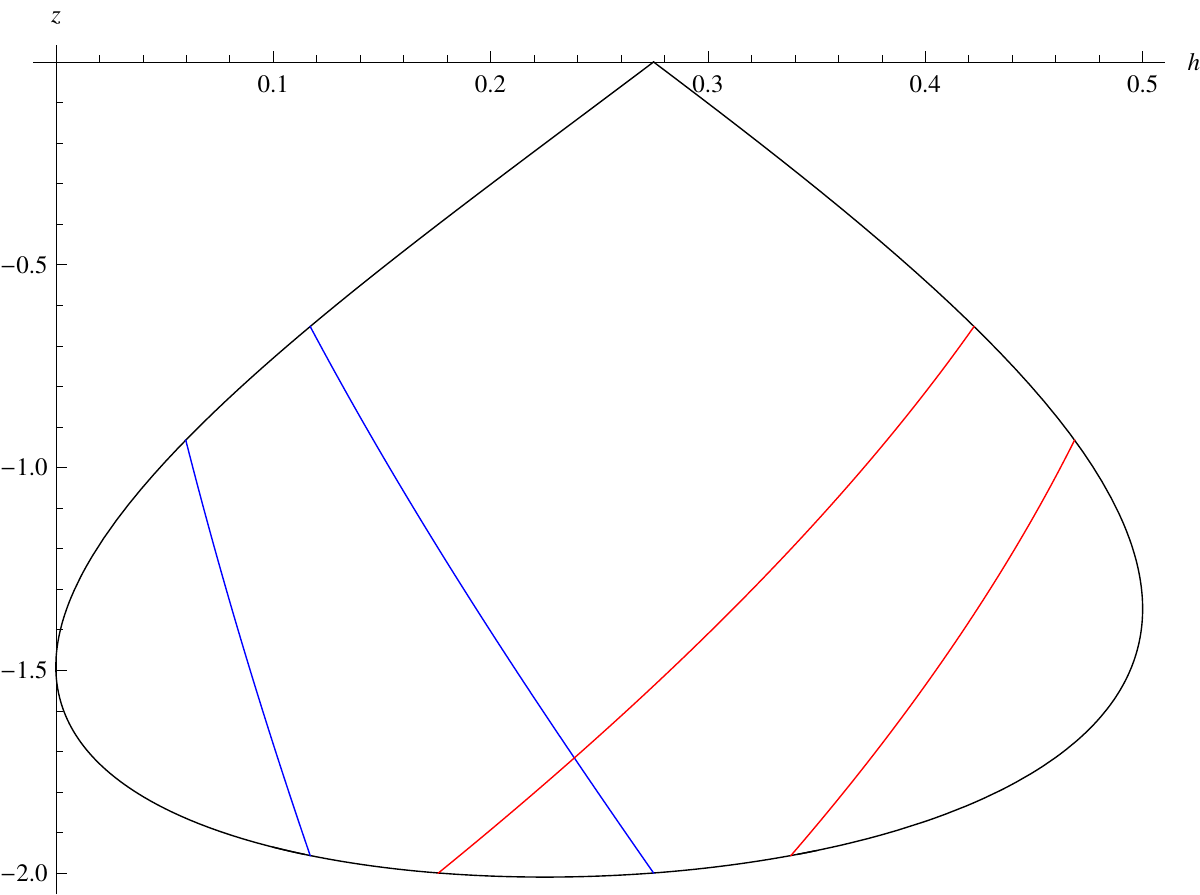} \label{triangular}} \quad
\subfigure[$7/12 < m < 2/3$]{\includegraphics[width=0.3\textwidth]{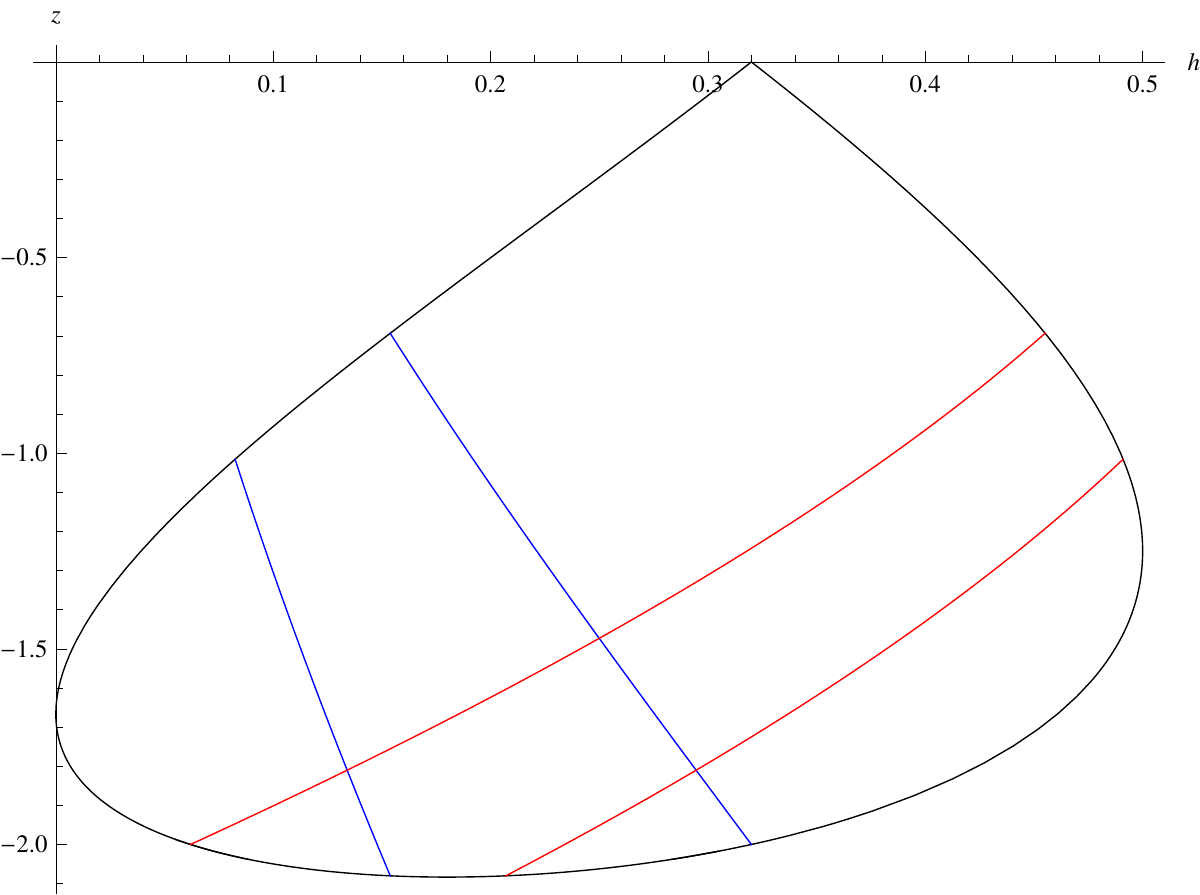} \label{pentagonal}} \quad
\subfigure[$2/3 \leq m < 1$]{\includegraphics[width=0.3\textwidth]{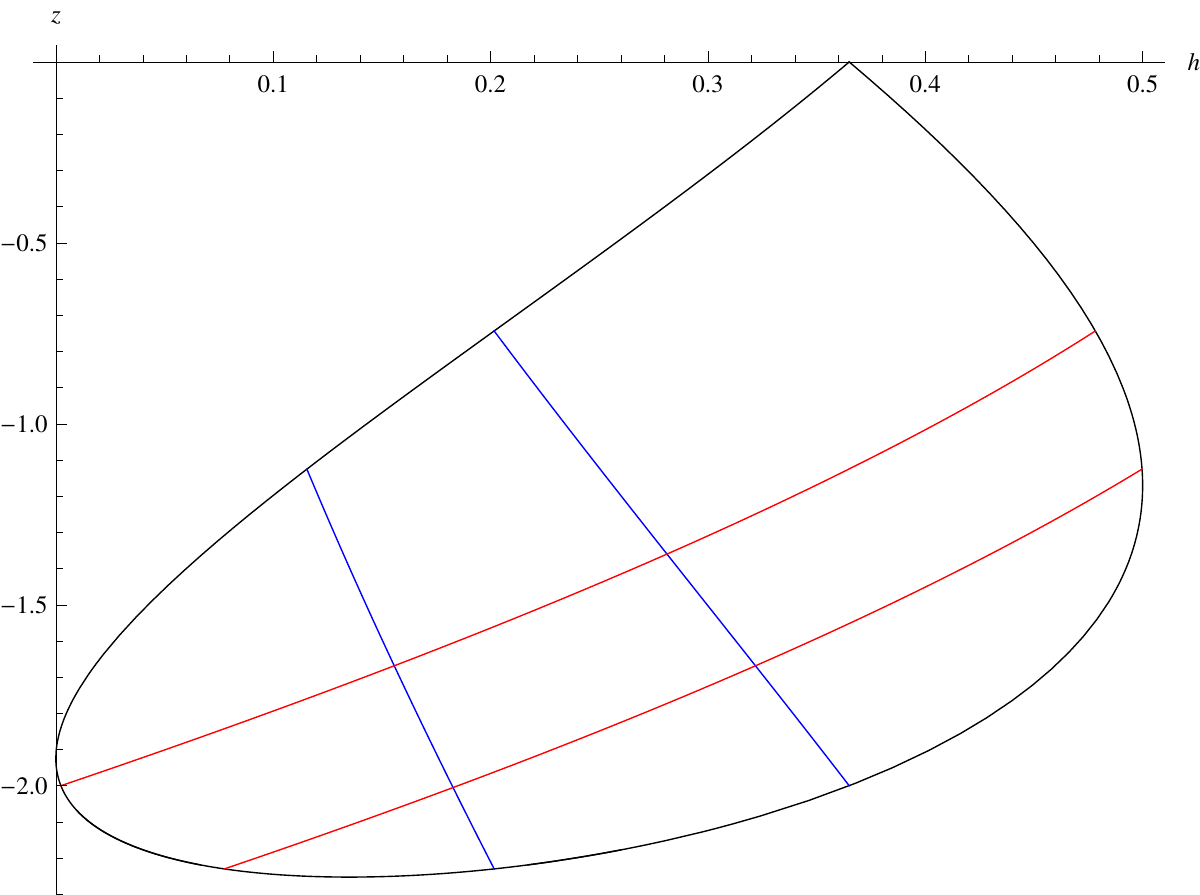} \label{quadrangular}}
\caption{The geometry of the sets $R_i \cap T(R_j)$ for $i,j \in \{0,1\}$.}
\label{RicapRj}
\end{figure}
We will work only with the case \ref{quadrangular}, when $2/3 \leq m$ and hence all four regions are quadrangular. The sets of points that correspond to the symbolic representations $\{ \underline{x} : x_0=0\}$ and $\{ \underline{x} : x_0=1\}$ are of course $R_0$ and $R_1$. We go one step further and identify the sets corresponding to symbolic representations $\{ \underline{x} : x_0,x_1 \in \{0,1\} \}$. Formally they are $R_i \cap T^{-1}(R_j)$ for $i,j \in \{0,1\}$, i.e. the preimages of the four quadrangular regions we have just discussed, but what is important is again the geometry of these sets, see Figure \ref{subrectangle}. They form (curvilinear) subrectangles
within $R_0$ or $R_1$, fully crossing them in the stable direction.
\begin{figure}
\centering
\includegraphics[width=0.75\textwidth]{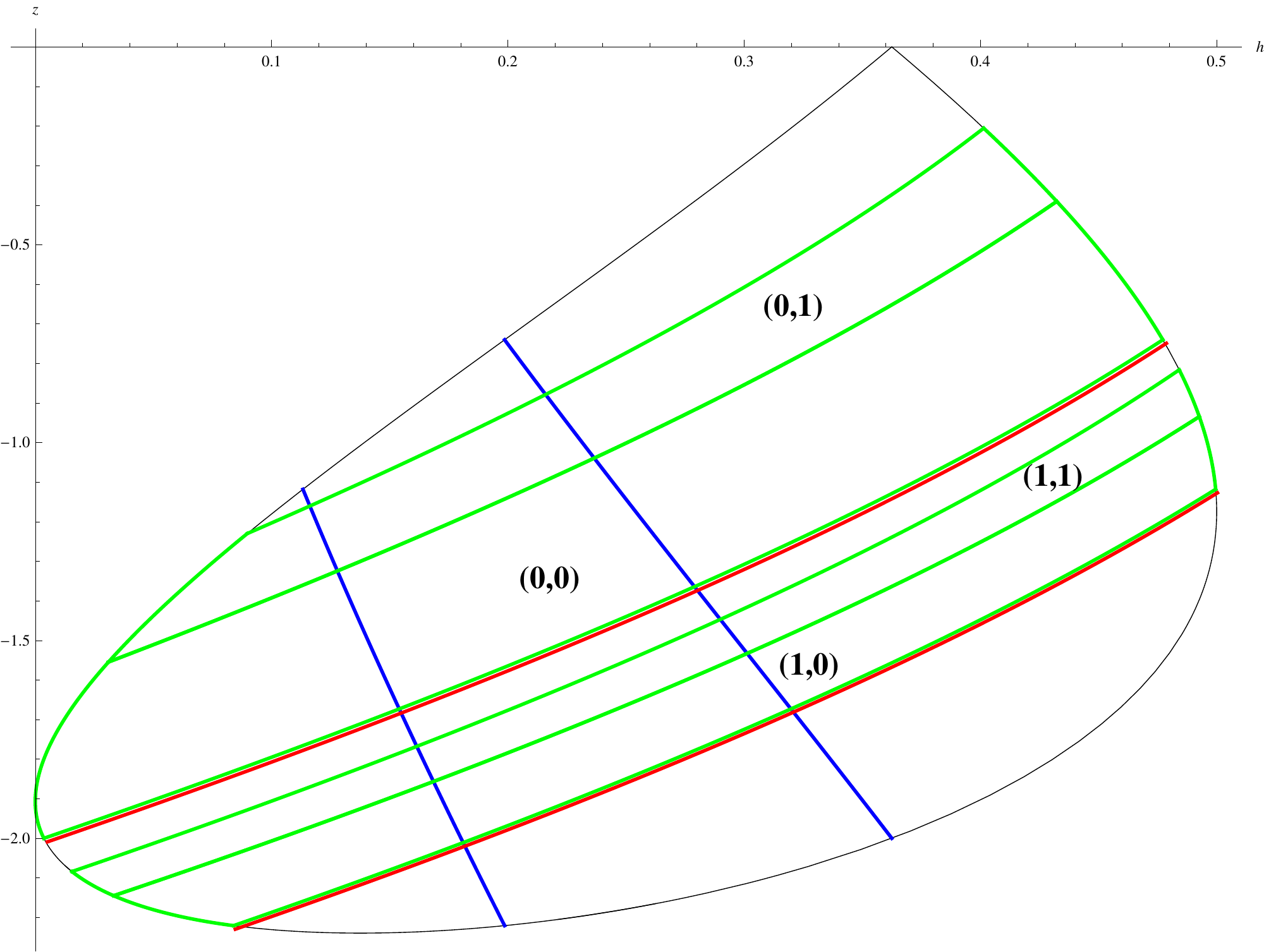}
\caption{The green rectangle with index $(i,j)$ is the set $R_i \cap T^{-1}(R_j)$.}
\label{subrectangle}
\end{figure}

Actually, in what follows we prove a somewhat stronger claim.

\textit{Claim} The stable sides of the rectangles $R_i \cap T^{-1}(R_j)$ for $i,j \in \{0,1\}$ cross both unstable sides of both of the rectangles $T(R_0)$ and $T(R_1)$.

In fact this holds automatically in case of the regions $R_1 \cap T^{-1}(R_0)$ and $R_1 \cap T^{-1}(R_1)$, because their stable sides are the preimages of segments of certain singularity curves -- $r_0$, $r_1$ and the segment of $\partial \mathcal{M}_1$ that forms the top edge of $T(R_1)$ -- which connect the two unstable sides of $T(R_1)$. Under the action of the inverse dynamics these unstable sides map onto the unstable sides of $R_1$, which are bits of the boundary of the phase space. Consequently, the above mentioned segments of the singularities map onto stable curves stretching from side to side in $\mathcal{M}_1$ and crossing the unstable sides of both $T(R_0)$ and $T(R_1)$.

The situation for the regions $R_0 \cap T^{-1}(R_0)$ and $R_0 \cap T^{-1}(R_1)$ is slightly more complicated, essentially because $R_0$ is topologically a triangle and so we can not talk about its stable and unstable sides. Actually it may happen that the top edge of $R_0 \cap T^{-1}(R_1)$ does not intersect the left side of $T(R_1)$ as required. To identify the cases when this happens we calculate the left endpoint of this top edge and compare it to the top left corner of $T(R_1)$. As follows from our discussion in section \ref{s:setup}, Formulas \eqref{pointsXBx}, \eqref{pointsIxIbx} and \eqref{TonRn}, the upper left corner of $R_0 \cap T^{-1}(R_1)$ is the point $T^{-1}(X_1(m))$, its second coordinate is $\frac{8(m-1)}{\sqrt{9-8m}}$. The upper left corner of $T(R_1)$ is the point $Ix_1(m)$, it has second coordinate $\frac{-2}{\sqrt{9-8m}}$. It follows then that for $m < 3/4$ the point $T^{-1}(X_1(m))$ falls outside the region $T(R_1)$. Since $R_0 \cap T^{-1}(R_1)$ lies above the region $R_0 \cap T^{-1}(R_0)$ this implies that for $2/3 \leq m \leq 3/4$ the stable sides of both regions cross the unstable sides of both $T(R_0)$ and $T(R_1)$. This completes the proof of the Claim.

If we want to proceed one more digit and identify the set of points corresponding to sequences with the first three digits arbitrarily chosen from $\{0,1\}$, we should consider the preimages of the sets $T(R_i) \cap R_j \cap T^{-1}(R_k)$ for $i,j,k \in \{0,1\}$. Note that due to our previous argument these sets form subrectangles in $T(R_0)$ and $T(R_1)$ fully crossing them in the stable direction, but narrower in the unstable direction (see again Figure \ref{subrectangle}). This, together with the previous geometric observations, implies that the preimage of any of these subrectangles is a subrectangle in one of the regions $R_i \cap T^{-1}(R_j)$ for $i,j \in \{0,1\}$ fully crossing it in the stable direction. It follows that the process can be iterated showing that for any $m \in [2/3, 3/4]$ and any given one-sided infinite sequence $x_0, x_1, \dots \in \{0,1\}$ the projection of the set $\{ \underline{y} : y_i = x_i \, \forall i \in \mathbb{N} \}$ is not empty. It is actually a curve, moreover a fairly long local stable manifold, that crosses the phase space.

To complete the proof it is enough to take into account that the involution $I : \mathcal{M}_1 \to \mathcal{M}_1$ maps the set $\cup_{i,j \in \{0,1\}}R_i \cap T(R_j)$ to itself and stable manifolds into unstable ones. The action of $I$ (more precisely the action of it lifted up) on the symbolic space is given by
\[ I(\underline{x})_i = x_{-i-1}. \]
Hence for any sequence $\underline{x} \in \{0,1\}^{\mathbb{Z}}$ the projection $\pi(\underline{x})$ is realised as a physical configuration. More than that it has long local stable and unstable manifolds
\[ W^s(\pi(\underline{x})) = \pi(\{ \underline{y} :  y_i = x_i \, \forall i \in \mathbb{N} \}), \quad W^u(\pi(\underline{x})) = \pi(\{ \underline{y} :  y_i = x_i \, \forall i < 0 \}), \]
crossing the phase space.
\end{proof}

\subsection{Convergence of period ratios}

With the help of this symbolic subsystem constructed in the previous subsection we can find the periodic points needed for our purposes. First of all consider the two special sequences, one consisting of all zeroes the other of all ones. Let us denote them by $\underline{0}$ and $\underline{1}$, respectively. Their projections are

\begin{equation}\label{fixedpointprojection}
\pi(\underline{0}) = F_0(m) \quad \pi(\underline{1}) = F_1(m),
\end{equation}
where the $F_i(m)$'s are the fixed points identified in \eqref{fixedpoints}. We now define our sequence of periodic points for any $m \in [2/3,3/4]$ as

\begin{equation}\label{periodicpoints}
\begin{split}
P_n(m) := \pi(\underline{p}^n) \text{ ,where } & p_i^n=0 \text{ for } -n \leq i \leq n-1,\\ & p_i^n=1 \text{ for } n \leq i \leq n+n^2-1 \\ & \text{and } \underline{p} \text{ is periodic otherwise.}
\end{split}
\end{equation}
Defined in this way $P_n(m)$ has discrete time period $n^2+2n$ and since
\[s(F_0(m), P_n(m)) = n\]
it is exponentially close to $F_0(m)$ by Formula \eqref{metricsequiv}. The local stable manifold of $P_n(m)$ intersects the local unstable manifold of $F_0(m)$ and the same holds with stable and unstable roles exchanged (this is actually true for any two points with symbolic representations from $\{0,1\}^{\mathbb{Z}}$). From this, together with the second half of Lemma \ref{subsystem}, it follows that for $n$ large enough the points $P_n(m)$ and $F_0(m)$ can be represented on the base of the same Young-tower. What remains is to show that in terms of period in continuous time $P_n(m)$ behaves more and more like $F_1(m)$ as $n$ increases. We make this precise in the next two propositions.

\begin{prop}\label{C0convergence}
Let $\tau_k$ denote the $k$-th Birkhoff sum of $\tau$. Then for any $m \in [2/3,3/4]$ we have
\[ \frac{\tau_{n^2+2n}(P_n(m),m)}{(n^2+2n) \tau(F_0(m),m)} \to \frac{\tau(F_1(m),m)}{\tau(F_0(m),m)} \]
as $n \to \infty$ and the convergence is uniform in $m$.
\end{prop}

\begin{proof}
First we show that $\tau$ is uniformly bounded on the set $(R_0 \cup R_1) \cap (T(R_0) \cup T(R_1))$, cf.~\eqref{taumax}.
Recall the formula for $\tau$ given in \eqref{dynamics}. On the phase space we have the following trivial bounds
\[ h \leq 1/2 \quad z \leq 0. \]
The relation between the first coordinate of the involution and the quantity $F$ (see \eqref{involformula}), together with the previous bound on $h$ implies that $F \leq 1/(2m)$. Finally on the set we are working on $n=0$ or $1$. Substituting the previous bounds into \eqref{dynamics} gives that

\begin{equation}\label{taumax}
\tau(h,z,m) \leq \frac{4}{\sqrt{m}} \leq \frac{4 \sqrt{3}}{\sqrt{2}}.
\end{equation}


We want to relate the value of the roof function $\tau$ along the trajectory of $P_n(m)$ to $\tau(F_1(m),m)$. By construction the trajectory of $P_n(m)$ gets exponentially close to $F_1(m)$ and stays close to this fixed point for an order of $n^2$ iterations. Hence to compare the two values we are going to perform a first order expansion of $\tau$ in the space variables (not in the parameter $m$).

The relevant partial derivatives are

\begin{equation}\label{partialderivatives}
\begin{split}
\frac{\partial \tau}{\partial h} = \frac{1}{\sqrt{2 m h}} + (2n+1) \frac{-1}{m \sqrt{2 F}} \\
\frac{\partial \tau}{\partial z} = 2(1-m) + (2n+1) \frac{2 \alpha z}{\sqrt{2 F}},
\end{split}
\end{equation}
where again $F$ and $\alpha$ are as in \eqref{Fandalpha}. To estimate these quantities on the set $(R_0 \cup R_1) \cap (T(R_0) \cup T(R_1))$ we use again that $n=0$ or $1$. Furthermore, we need lower bounds on $h$ and $F$ (because they appear in some denominators) and also on $z$ (because $\alpha \leq 0$ for $m \in [1/2,1]$). Now a lower bound on $h$ can be the first coordinate of the leftmost point of our domain, which is $Ixh_1(m) = \frac{m}{2(9-8m)}$. Using again the relation in \eqref{involformula}, together with the fact that our domain is mapped onto itself by $I$, we get that $F \geq \frac{1}{2(9-8m)}$. Finally the minimum of $z$ is given by the second coordinate of the lowest point of the domain. Since this is a point of $r_1$, which is an increasing curve, it can be further estimated by the second coordinate of the left endpoint of $r_1$, which is $Bxz_1(m) = \frac{-3}{\sqrt{4-3m}}$. It is then clear that all estimates are continuous functions of $m$ on $[2/3,3/4]$ (actually on the whole parameter domain $[1/2,1]$) and all denominators in \eqref{partialderivatives} are separated from $0$, hence there is a uniform bound on $\left\| \nabla_{h,z} \tau \right\|$. Let us denote this bound by $G_{\infty}$.

Now by the mean value theorem the continuous time period of $P_n(m)$ can be expressed as

\begin{equation}\label{taufirstexpand}
\begin{split}
\tau_{n^2+2n}(P_n(m),m) = & \tau_{n}(P_n(m),m) + \tau_{n}(T^{n^2+n}P_n(m),m) + n^2 \tau(F_1(m),m) + \\ & + \sum\limits_{i=n}^{n^2+n-1} D_{\underline{v}_i}\tau(\xi_i,m) \cdot d(F_1(m), T^iP_n(m)).
\end{split}
\end{equation}
Here $\underline{v}_i$ is the unit vector parallel to the line connecting $F_1(m)$ and $T^iP_n(m)$, $D_{\underline{v}_i}$ denotes differentiating in the $\underline{v}_i$ direction and $\xi_i$ is the point on the segment specified by the mean value theorem. Apart from the bounds \eqref{taumax} and those given on the gradient of $\tau$ we will also estimate $d(F_1(m), T^iP_n(m))$ using \eqref{metricsequiv}. For this note that
\[ s(F_1(m), T^iP_n(m)) = \min\{i-n, n^2+n-i\} \quad \text{for } n \leq i \leq n^2+n, \]
so the sum of the distances in \eqref{taufirstexpand} can be estimated by segments of two convergent geometric series (one for the indices $n \leq i \leq n^2/2+n$ and the other one for $n^2/2+n \leq i \leq n^2+n-1$), both with quotient $\theta$. Hence we have the overall estimate

\begin{multline*}
\left| \frac{\tau_{n^2+2n}(P_n(m),m)}{(n^2+2n) \tau(F_0(m),m)} - \frac{\tau(F_1(m),m)}{\tau(F_0(m),m)} \right| \leq \\
\leq \left| \frac{n^2 \tau(F_1(m),m)}{(n^2+2n) \tau(F_0(m),m)} - \frac{\tau(F_1(m),m)}{\tau(F_0(m),m)} \right| + \\
+ \frac{2n \cdot 4 \sqrt{3/2}}{(n^2+2n) \tau(F_0(m),m)} + \frac{G_{\infty} 2C (1-\theta^{n^2/2+1})}{(n^2+2n) \tau(F_0(m),m) (1-\theta)},
\end{multline*}

from which the statement of the proposition immediately follows.
\end{proof}

To complete our argument we have to show that not only the ratio of the two periods converge but also its derivative with respect to the parameter $m$, as expressed in the following Proposition.

\begin{prop}\label{C1convergence}
For any $m \in [2/3,3/4]$ we have
\[\frac{d}{dm}\left( \frac{\tau_{n^2+2n}(P_n(m),m)}{(n^2+2n) \tau(F_0(m),m)} \right) \to \frac{d}{dm}\left( \frac{\tau(F_1(m),m)}{\tau(F_0(m),m)} \right)\]
as $n \to \infty$ and the convergence is uniform in $m$.
\end{prop}

The proof of Proposition~\ref{C1convergence} requires more work, hence we move it to subsection~\ref{s:technical}.

Now we are in a position to prove our first result.

\begin{proof}[Proof of Theorem \ref{aediophantine}]
We have the sequence of special periodic points $P_n(m)$ converging to $F_0(m)$ for every $m \in [2/3,3/4]$. As stated in Theorem \ref{youngtower}, in \cite{BBNV} we showed the existence of an open interval in $(1/2,3/4)$ such that for every $m \in I$ the discrete time map $T : \mathcal{M}_1 \to \mathcal{M}_1$ can be modelled by a Young tower with exponential tails. The size of such a tower is determined by certain constants of the system, for example uniform bounds on the curvatures of unstable manifolds and of singularity curves, and also the minimum rates of expansion and contraction. What is important is that for $m \in [2/3,3/4]$ these constants can be chosen independently of $m$, since the dynamics is $C^2$ and the dependence on $m$ is continuous. Then we have a tower with uniform size, hence for $n$ large enough $P_n(m)$ and $F_0(m)$ are close enough to each other and so, as we have shown, they can be represented on the same Young-tower for any value of $m$. Applying the result of Melbourne (\cite{M}) it is then enough to prove that the ratio of the continuous time periods of $P_n(m)$ and $F_0(m)$ (the latter considered under $(n^2+2n)$ iterations of $T$) is Diophantine. Propositions \ref{C0convergence} and \ref{C1convergence} imply that this ratio as a function of $m$ converges in the $C^1$ topology to the function in \eqref{periodratio}. The limit is a $C^1$, strictly decreasing function for $m \in [2/3,3/4]$, hence by further increasing the value of $n$ (if necessary) the ratio $\tau_{n^2+2n}(P_n(m),m)/((n^2+2n)\tau(F_0(m),m))$ will also be strictly decreasing and therefore indeed Diophantine for almost every $m$ in this interval.
\end{proof}

\subsection{Proof of Proposition \ref{C1convergence} \label{s:technical}}

Differentiating the fractions appearing in the statement of the Proposition we get

\begin{gather*}\label{ratioderivatives}
\frac{d}{dm}\left( \frac{\tau_{n^2+2n}(P_n(m),m)}{(n^2+2n) \tau(F_0(m),m)} \right) = \\ = \frac{\tau(F_0(m),m) \frac{d}{dm}\tau_{n^2+2n}(P_n(m),m) - \tau_{n^2+2n}(P_n(m),m) \frac{d}{dm}\tau(F_0(m),m)}{(n^2+2n)\tau^2(F_0(m),m)} \\\\
\frac{d}{dm}\left( \frac{\tau(F_1(m),m)}{\tau(F_0(m),m)} \right) = \\ = \frac{\tau(F_0(m),m) \frac{d}{dm}\tau(F_1(m),m) - \tau(F_1(m),m) \frac{d}{dm}\tau(F_0(m),m)}{\tau^2(F_0(m),m)}.
\end{gather*}
Hence, by Proposition \ref{C0convergence} it is enough to prove that

\begin{prop}\label{C1convsimplified}
\[\frac{1}{n^2+2n}\frac{d}{dm}\tau_{n^2+2n}(P_n(m),m) \to \frac{d}{dm}\tau(F_1(m),m),\]
as $n \to \infty$.
\end{prop}
To show this we first derive a useful formula for the derivative. Before that let us introduce some notations for brevity. For any $k \in \{0, \dots, n^2+2n-1\}$ let
\[P_n^k(m) := T^kP_n(m),\]
so that for any integer $j$ we have $T^jP_n(m) = P_n^l(m)$, where $l \equiv j\, (mod \; n^2+2n)$. Also we will refer to the two dimensional vector obtained by differentiating the coordinates of $P_n^i(m)$ with respect to $m$ as $\frac{d}{dm}P_n^i(m)$, and $\Bigl(\frac{d}{dm}P_n^i(m),1\Bigr)$ will denote the three dimensional vector that has first two coordinates identical to $\frac{d}{dm}P_n^i(m)$ and $1$ as the third coordinate. Using this notation we can write

\begin{equation}\label{roofsumderiv}
\begin{split}
\frac{d}{dm}\tau_{n^2+2n}(P_n(m),m) & = \sum\limits_{i=0}^{n^2+2n-1} \frac{d}{dm}\tau(P_n^i(m),m) \\ & = \sum\limits_{i=0}^{n^2+2n-1} \left\langle (\nabla_{h,z,m}\tau)(P_n^i(m),m), \Bigl(\frac{d}{dm}P_n^i(m),1\Bigr) \right\rangle,
\end{split}
\end{equation}
where $\langle.,.\rangle$ denotes the usual scalar product in $\mathbb{R}^3$. While we have an explicit formula for $\tau$ and hence also for $\nabla_{h,z,m}\tau$, calculating the two dimensional vector $\frac{d}{dm}P_n^i(m)$ is slightly more difficult. For this we use that
\[T^{n^2+2n}P_n^k(m)=P_n^k(m),\]
and we perform implicit differentiation with respect to $m$ keeping in mind that the dynamics $T$ also depends on $m$. Let us denote by $\frac{\partial T}{\partial m}(P_n^{k+i})$ the two dimensional vector obtained by first differentiating $T$ with respect to $m$ and then evaluating the result at the point $P_n^{k+i}$. As we perturb the parameter we have to take into account that not just every point of the trajectory varies with $m$, but the dynamics also chenges. These two effects jointly appear in the calculations leading to the formula

\begin{multline}\label{periodicderiv}
\frac{d}{dm}P_n^k(m) = \\= \left(I-\prod\limits_{i=0}^{n^2+2n-1}DT(P_n^{k+i}(m))\right)^{-1} \sum\limits_{i=0}^{n^2+2n-1} \prod\limits_{j=i+1}^{n^2+2n-1}DT(P_n^{k+j}(m)) \cdot \frac{\partial T}{\partial m}(P_n^{k+i}(m)) \\= \left(I-DT^{n^2+2n}(P_n^k(m))\right)^{-1} \sum\limits_{i=0}^{n^2+2n-1} DT^{n^2+2n-(i+1)}(P_n^{k+i+1}(m)) \cdot \frac{\partial T}{\partial m}(P_n^{k+i}(m)) \\= \sum\limits_{i=0}^{n^2+2n-1} (DT^{(i+1)-(n^2+2n)}(P_n^k(m)) - DT^{i+1}(P_n^k(m)))^{-1} \cdot \frac{\partial T}{\partial m}(P_n^{k+i}(m)),
\end{multline}
where in the last line we used the inverse differentiation rule and again the fact that the point $P_n^k(m)$ is periodic with period $n^2+2n$. In the next step we show that this sum is bounded.

\begin{lem}\label{boundedderiv}
For each $k \in \{0, \dots, n^2+2n-1\}$ the quantity $\Bigl|\frac{d}{dm}P_n^k(m)\Bigr|$ is uniformly bounded.
\end{lem}

\begin{proof}
To make the notations simpler in this proof we will suppress the dependence of the objects on the parameter $m$. Let us denote the $i$-th term in the sum \eqref{periodicderiv} by $v_i(n,k)$, so

\begin{equation}\label{termrearranged}
\frac{\partial T}{\partial m}(P_n^{k+i}) = (DT^{(i+1)-(n^2+2n)}(P_n^k) - DT^{i+1}(P_n^k)) \cdot v_i(n,k).
\end{equation}
Observe that because $P_n^k$ is a periodic point the tangent spaces of its stable and unstable manifolds can be calculated as the stable and unstable eigendirections of the tangent map $DT^{n^2+2n}(P_n^k)$, respectively. We denote the \emph{normalized} eigenvectors of this matrix by $s(n,k)$ and $u(n,k)$ where the letters $u$ an $s$  refer to stable and unstable, respectively. Then by the invariance of these directions we have
\[DT(P_n^k) s(n,k) = \mu(n,k) s(n,k+1), \quad DT(P_n^k) u(n,k) = \lambda(n,k) u(n,k+1),\]
defining the quantities $\lambda(n,k) < -1 < \mu(n,k) < 0$, (both are negative as the tangent map contains a rotation by $180$ degrees). Here and throughout the subsection the index $k$ should be understood modulo $n^2+2n$. Note that $\prod\limits_{k=0}^{n^2+2n-1} \mu(n,k)$ and $\prod\limits_{k=0}^{n^2+2n-1} \lambda(n,k)$ are the stable and unstable eigenvalues of the matrix $DT^{n^2+2n}(P_n^k)$ and hence their product is $1$, since $DT^{n^2+2n}(P_n^k)$ has determinant $1$. We consider the decompositions
\[v_i(n,k) = a_i s(n,k) + b_i u(n,k), \quad \frac{\partial T}{\partial m}(P_n^{k+i}) = c_i s(n,k+i+1) + d_i u(n,k+i+1),\]
and substitute them into \eqref{termrearranged} to get

\begin{multline*}
c_i s(n,k+i+1) + d_i u(n,k+i+1) = \\ = a_i \prod\limits_{j=1}^{n^2+2n-(i+1)} \frac{1}{\mu(n,k-j)} s(n,k+i+1) \ +\  b_i \prod\limits_{j=1}^{n^2+2n-(i+1)} \frac{1}{\lambda(n,k-j)} u(n,k+i+1) \ - \\ - a_i \prod\limits_{j=0}^{i} \mu(n,k+j) s(n,k+i+1) - b_i \prod\limits_{j=0}^{i} \lambda(n,k+j) u(n,k+i+1) = \\ = a_i \prod\limits_{j=0}^i \mu(n,k+j) \biggl(\prod\limits_{k=1}^{n^2+2n} \lambda(n,k) -1\biggr) s(n,k+i+1) \ + \\ + b_i \prod\limits_{j=0}^i \lambda(n,k+j) \biggl(\prod\limits_{k=1}^{n^2+2n} \mu(n,k) -1\biggr) u(n,k+i+1).
\end{multline*}

This gives the relations between the coefficients

\begin{equation}\label{coeffrelation}
\begin{split}
a_i = \frac{c_i}{\prod\limits_{j=0}^{i} \mu(n,k+j)} \frac{1}{\prod\limits_{k=1}^{n^2+2n} \lambda(n,k) -1} \\
b_i = \frac{d_i}{\prod\limits_{j=0}^{i} \lambda(n,k+j)} \frac{1}{\prod\limits_{k=1}^{n^2+2n} \mu(n,k) -1}.
\end{split}
\end{equation}
First we show that both $c_i$ and $d_i$ are uniformly bounded. To see this, by the uniform transversality of stable and unstable cones it is enough to check that $\left\| \frac{\partial T}{\partial m}(P_n^{k+i}) \right\|$ is itself bounded. We have the general formula for the derivative of \eqref{dynamics}
\[ \frac{\partial T}{\partial m}(h,z,m) = \Bigl( 1+z^2(1-6m+6m^2), \frac{-(2n+2)}{\sqrt{2 F}}(\frac{h}{m^2}+z^2(4m-3))\Bigr), \]
and as we have already shown right after \eqref{partialderivatives}, on the domain $(R_0 \cup R_1) \cap (T(R_0) \cup T(R_1))$ and for the parameter interval $m \in [2/3,3/4]$ the quantity $F$ is bounded away from $0$, while $h$ and $z^2$ is bounded, hence the whole norm $\left\| \frac{\partial T}{\partial m}(P_n^{k+i}) \right\|$ is bounded, too.
\\Finally, since the dynamics is uniformly hyperbolic for all fixed values of $m$ (proved in \cite{BBNV}), there are numbers $\mu$ and $\lambda$ such that $\lambda(n,k) \leq \lambda < -1 < \mu \leq \mu(n,k) < 0$ for all $k$. Hence both sets of numbers $\{a_{n^2+2n-1-i}\}_{i=0}^{n^2+2n-1}$ (i.e. the $a_i$'s in reversed order) and $\{b_i\}_{i=0}^{n^2+2n-1}$ can be estimated by the initial segment of a geometric series. Indeed this is immediate for the $b_i$'s, and for the $a_i$'s note that when the index $i$ is large and hence the denominator $\prod\limits_{j=0}^{i} \mu(n,k+j)$ is very small, it is still compensated by $\prod\limits_{k=1}^{n^2+2n} \lambda(n,k)$. Therefore these sequences are summable and so elementary inequalities imply that the statement of the lemma is true.
\end{proof}

Now we are in the position to prove Proposition \ref{C1convsimplified}.

\begin{proof}[Proof of Proposition \ref{C1convsimplified}]
Consider the formula \eqref{roofsumderiv} we gave for the derivative on the left hand side of Proposition \ref{C1convsimplified}. By Lemma \ref{boundedderiv} we know that each vector $(\frac{d}{dm}P_n^i(m),1)$ is bounded and we have already shown after \eqref{partialderivatives} that the gradient of $\tau$ is also bounded, however this was the gradient only in the spatial variables. In \eqref{roofsumderiv} the gradient contains the derivative with respect to $m$, too. By \eqref{dynamics}
\[\frac{\partial \tau}{\partial m} = \frac{2n+1}{\sqrt{2 F}} \Bigl(\frac{h}{m^2} + (4m-3)z^2\Bigr) - \sqrt{\frac{h}{2m^3}} - 2z,\]
and basically the same argument given after \eqref{partialderivatives} shows that this quantity and hence the whole gradient in each of the terms of \eqref{roofsumderiv} is bounded. Therefore each scalar product in \eqref{roofsumderiv} is bounded as well. In \eqref{roofsumderiv} one can replace $P_n(m)$ by $F_1(m)$ to obtain a similar formula for $\frac{d \tau(F_1(m),m)}{dm}$. After this

\begin{multline}\label{middlepart}
\frac{1}{n^2+2n} \frac{d \tau_{n^2+2n}(P_n(m),m)}{dm} - \frac{d \tau(F_1(m),m)}{dm} = \\
= \frac{1}{n^2+2n} \sum\limits_{i=0}^{n^2+2n-1} \left\langle (\nabla_{h,z,m}\tau)(P_n^i(m),m), \biggl(\frac{d}{dm}P_n^i(m),1\biggr) \right\rangle - \frac{d \tau(F_1(m),m)}{dm} = \\
= \frac{1}{n^2+2n} \sum\limits_{i=3n}^{n^2-n} \left\langle (\nabla_{h,z,m}\tau)(P_n^i(m),m) - (\nabla_{h,z,m}\tau)(F_1(m),m), \biggl(\frac{d}{dm}P_n^i(m),1\biggr) \right\rangle + \\
+ \frac{1}{n^2+2n} \sum\limits_{i=3n}^{n^2-n} \left\langle (\nabla_{h,z,m}\tau)(F_1(m),m), \biggl(\frac{d}{dm}P_n^i(m) - \frac{d}{dm}F_1(m),0\biggr) \right\rangle + \\ + \mathcal{O}\Bigl(\frac{1}{n}\Bigr).
\end{multline}

Note that we compressed an order of $n$ number of summands into the error term and so we are left to deal only with the middle part of the original sum. The first sum is easy to handle. Basically it is enough to check that $\nabla_{h,z,m}\tau$ is $C^1$ in the variables $(h,z)$, which turns out to be the case after doing similar computations as before. After this an argument similar to the one used in the proof of Proposition \ref{C0convergence} shows that the first sum in \eqref{middlepart} tends to $0$ as $n \to \infty$. Actually it is exponentially small in $n$, because for $3n \leq i \leq n^2-n$ the sum of the distances between the points $P_n^i(m)$ and $F_1(m)$ can be estimated by a segment of a geometric series with the largest term being exponentially small, since $s(P_n^i(m),F_1(m)) \geq 2n$ for such indices.
\\What remains is to show that the quantity

\begin{equation}\label{lastsum}
\frac{1}{n^2+2n} \sum\limits_{i=3n}^{n^2-n} \left\langle (\nabla_{h,z,m}\tau)(F_1(m),m), \biggl(\frac{d}{dm}P_n^i(m) - \frac{d}{dm}F_1(m),0\biggr) \right\rangle
\end{equation}
tends to $0$ as $n\to \infty$. We know that the gradient $(\nabla_{h,z,m}\tau)(F_1(m),m)$ is bounded so we estimate the norm of the other vector, which is a difference of two derivatives. We use the explicit formula \eqref{periodicderiv} for these differentials to get

\begin{multline}\label{derivdiffer}
\frac{d}{dm}P_n^i(m) - \frac{d}{dm}F_1(m) =\\= \sum\limits_{j=0}^{n^2+2n-1}(DT^{(j+1)-(n^2+2n)}(P_n^i(m)) - DT^{j+1}(P_n^i(m)))^{-1} \cdot \frac{\partial T}{\partial m}(P_n^{j+i}(m)) -\\- (DT^{(j+1)-(n^2+2n)}(F_1(m)) - DT^{j+1}(F_1(m)))^{-1} \cdot \frac{\partial T}{\partial m}(F_1(m)).
\end{multline}
To estimate the middle part of this sum we use ideas from the proof of Lemma \ref{boundedderiv} (especially formula \eqref{coeffrelation}), which remains valid and actually is simpler, when $P_n^i(m)$ is replaced by $F_1(m)$. From there it follows that for $n \leq j \leq n^2+n-1$ the terms that form the $j$-th difference in \eqref{derivdiffer} are already exponentially small in $n$ so we can omit them. Therefore it is enough to estimate the differences in \eqref{derivdiffer} for $3n \leq i \leq n^2-n$, $0 \leq j \leq n-1$ and $n^2+n \leq j \leq n^2+2n-1$. Using the triangular inequality we estimate the norm of the $j$-th term by

\begin{multline}\label{usualtrick}
\biggl\|(DT^{(j+1)-(n^2+2n)}(F_1(m)) - DT^{j+1}(F_1(m)))^{-1} \cdot \biggl(\frac{\partial T}{\partial m}(P_n^{j+i}(m)) - \frac{\partial T}{\partial m}(F_1(m))\biggr)\biggr\| + \\ + \biggl\|\biggl[(DT^{(j+1)-(n^2+2n)}(P_n^i(m)) - DT^{j+1}(P_n^i(m)))^{-1} -\\- (DT^{(j+1)-(n^2+2n)}(F_1(m)) - DT^{j+1}(F_1(m)))^{-1}\biggr] \cdot \frac{\partial T}{\partial m}(P_n^{j+i}(m))\biggr\|.
\end{multline}
Let us denote the vector in the first term by $w_{j,i}$. Note that $\frac{\partial T}{\partial m}$ is $C^1$ on the domain we work on, so after setting
\[v_{j,i} := \frac{\partial T}{\partial m}(P_n^{j+i}(m)) - \frac{\partial T}{\partial m}(F_1(m))\]
we have

\begin{equation}\label{vjinorm}
\left\|v_{j,i}\right\|  \leq C \cdot d(P_n^{j+i}(m),F_1(m)) \leq \bar{C} \theta^{s(P_n^{j+i}(m),F_1(m))}.
\end{equation}
We denote the coordinates of $v_{j,i}$ and $w_{j,i}$ in the basis $\{s(F_1(m)),u(F_1(m))\}$ by $(a_{j,i},b_{j,i})$ and $(c_{j,i},d_{j,i})$, respectively. Then a calculation analogous to the one in the proof of Lemma \ref{boundedderiv} (c.f. \eqref{coeffrelation}) gives

\begin{equation}\label{coeffrelationagain}
\begin{split}
c_{j,i} = \frac{a_{j,i} \cdot \lambda(F_1(m))^{j+1}}{\lambda(F_1(m))^{n^2+2n} -1} \\
d_{j,i} = \frac{b_{j,i} \cdot \mu(F_1(m))^{j+1}}{\mu(F_1(m))^{n^2+2n} -1}.
\end{split}
\end{equation}
Due to the uniform transversality of stable and unstable cones, estimates on the norm of $w_{j,i}$ are -- up to a constant -- the same as estimates on the coordinates $c_{j,i}, d_{j,i}$. By the same reason \eqref{vjinorm} implies that
\[|a_{j,i}| \leq C \cdot \theta^{s(P_n^{j+i}(m),F_1(m))}, \quad |b_{j,i}| \leq C \cdot \theta^{s(P_n^{j+i}(m),F_1(m))},\]
in particular they are bounded and the estimates form segments of geometric series in $j$. Hence the sum of $c_{j,i}$'s for $0 \leq j \leq n-1$ and the sum of $d_{j,i}$'s for $n^2+n \leq j \leq n^2+2n-1$ are both exponentially small in $n$. To see that this holds also for the sum of $c_{j,i}$'s for $n^2+n \leq j \leq n^2+2n-1$ and the sum of $d_{j,i}$'s for $0 \leq j \leq n-1$ use in addition that $3n \leq i \leq n^2-n$ and hence $s(P_n^{j+i}(m),F_1(m)) \geq n$. These observations altogether leads to the fact that for all $i$ fixed between $3n$ and $n^2-n$
\[\sum\limits_{j=0}^{n-1}\|w_{j,i}\| + \sum\limits_{j=n^2+n}^{n^2+2n-1}\|w_{j,i}\| \leq C \cdot \theta^n.\]
Therefore it remains only to estimate the second term in \eqref{usualtrick}. For this we redefine the coefficients $a_{j,i}, \dots, d_{j,i}$ as

\begin{equation}\label{redefine}
\begin{split}
\frac{\partial T}{\partial m}(P_n^{j+i}(m)) = a_{j,i} s(P_n^i(m)) + b_{j,i} u(P_n^i(m)) \\
\frac{\partial T}{\partial m}(P_n^{j+i}(m)) = c_{j,i} s(F_1(m)) + d_{j,i} u(F_1(m)).
\end{split}
\end{equation}
In this notation, after a calculation analogous to the repeatedly referred one from Lemma \ref{boundedderiv}, the second term of \eqref{usualtrick} reads as

\begin{multline}\label{dynholder}
\biggl\|\frac{a_{-1,i+j+1}}{\prod\limits_{k=0}^j\mu(P_n^{i+k}(m))} \frac{1}{\prod\limits_{k=1}^{n^2+2n}\lambda(P_n^k(m)) -1} s(P_n^i(m)) + \\ + \frac{b_{-1,i+j+1}}{\prod\limits_{k=0}^j\lambda(P_n^{i+k}(m))} \frac{1}{\prod\limits_{k=1}^{n^2+2n}\mu(P_n^k(m)) -1} u(P_n^i(m)) - \\ - \frac{c_{j,i} \cdot \lambda(F_1(m))^{j+1}}{\lambda(F_1(m))^{n^2+2n}-1} s(F_1(m)) - \frac{d_{j,i} \cdot \mu(F_1(m))^{j+1}}{\mu(F_1(m))^{n^2+2n}-1} u(F_1(m))\biggr\|.
\end{multline}
To finish the proof we use the regularity of the stable and unstable directions stated in the following lemma.

\begin{lem}\label{stableunstabledhc}
Let $x$ and $y$ be two points with both stable and unstable directions well-defined. Further assume that for all $i$ for which $-s_-(x,y) < i < s_+(x,y)$, the points $T^i(x)$ and $T^i(y)$ are either in $R_0$ or in $R_1$. Then there exist constants $C_s,C_u >0$ and $\gamma_s,\gamma_u \in (0,1)$ such that
\[\|s(x)-s(y)\| \leq C_s \gamma_s^{s(x,y)} \quad \|u(x)-u(y)\| \leq C_u \gamma_u^{s(x,y)}.\]
\end{lem}
This Lemma expresses the dynamical H\"older continuity of the stable/unstable directions, but we were unable to find a good reference, so instead we give a proof of the Lemma in the Appendix. As a consequence, using also that $\|DT\|$ is bounded on our domain, this property holds for the functions $\mu(.)$ and $\lambda(.)$ as well. Finally note that all coefficients $a_{j,i}, \dots, d_{j,i}$ can be expressed using scalar products of the vectors $u(.)$, $s(.)$, their orthocomplements and $\frac{\partial T}{\partial m}$. We have already shown that the latter is uniformly bounded and this, together with the uniform transversality of stable and unstable cones, leads to the fact that both differences $a_{-1,i+j+1} - c_{j,i}$ and $b_{-1,i+j+1} - d_{j,i}$ are exponentially small. It is then a straightforward calculation, using several triangular inequalities, that \eqref{dynholder} is exponentially small and hence so is the second term in \eqref{usualtrick}. This completes the proof of Propositions~\ref{C1convsimplified} and thus~\ref{C1convergence}.
\end{proof}

\section{Extension of results \label{s:extension}}

As the reader might have already noticed the calculations in Section \ref{calculations} do not depend on the system in hand crucially. The system specific parts of the argument were
\begin{enumerate}
	\item to find a full shift subsystem based on the geometrical properties,
	\item to give bounds on certain quantites like $\tau$ and its derivatives, or the expansion and contraction rates of $T$,
	\item to show that the ratio of the continuous periods of the original fixed points for $T$ is a $C^1$ function of the parameter $m$ with nonzero derivative.
\end{enumerate}
Appart from these the whole method works in full generality. Taking advantage of this observation, in this section we extend our results for a larger set of parameters.
\\As for $2.$, without going into details or doing the actual calculations, we claim that whenever we work on a domain of the form $\cup_{i,j \in \{k-1,k\}} (R_i \cap T(R_j))$, all the mentioned quantities will be uniformly bounded for the corresponding interval of the parameter. This can be checked in an analogous way as we did in Section \ref{calculations} for $k=1$.
\\In Lemma \ref{subsystem} we discussed the geometry of the sets $R_i \cap T(R_j)$ for $i,j \in \{0,1\}$ and showed that their preimages intersect the unstable sides of the sets $T(R_0)$ and $T(R_1)$ if $m \in [2/3,3/4]$. From this we could conclude the existence of a subsystem that is a full shift on the two symbols $0$ and $1$. Now we can essentially repeat our argument for the stripes $R_{k-1}, R_k$ with larger $k$ giving a similar full shift subsystem for different parameter intervals. This is even simpler than in Lemma \ref{subsystem}, because for $k \geq 1$ all stripes $R_k$ are topologically squares, in contrast to $R_0$ that is topologically a triangle, which required special care.

\begin{lem}\label{subsystemextended}
Given any $k \geq 1$, for every $m \in \bigl[\frac{1+k}{2+k}, \frac{2k^2+2k-1}{2k^2+2k}\bigr]$ the projection of any sequence $\underline{x} \in \{k-1,k\}^{\mathbb{Z}}$ is realised as a physical configuration, i.e. as a point in $\mathcal{M}_1$. Any such point has local stable and unstable manifolds that fully cross the phase space.
\end{lem}

\begin{proof}
The $k=1$ case was proved in Lemma \ref{subsystem}. To apply the same argument to the case $k \geq 2$, one has to check only that the sets $R_i \cap T(R_j)$ for $i,j \in \{k-1,k\}$ are all quadrangular. The required intersections and hence the whole iteration scheme is then automatic. Also, checking that all the previous four sets are quadrangular is easy. The geometry of the system, in particular that the singularities are all increasing, while the inverse singularities are all decreasing curves, implies that it is enough to show that $I(r_k)$ intersects both $r_k$ and $r_{k-2}$. A nonempty intersection $r_k \cap I(r_k)$ can be guaranteed, using the mentioned monotonicity of the curves, by showing that the left endpoint of $r_k$ has smaller $h$ coordinate than the right endpoint of $I(r_k)$. Based on \eqref{pointsXBx} and \eqref{pointsIxIbx} this condition is equivalent to

\begin{gather*}
Bxh_k(m) \leq Ibxh_k(m) \Leftrightarrow \\ \Leftrightarrow \frac{m(-2m(k+2)+2k+3)^2}{2(1-m)k(k+2)+2} \leq \frac{m}{2(1-m)k(k+2)+2} \Leftrightarrow \\ \Leftrightarrow \frac{1+k}{2+k} \leq m.
\end{gather*}
For the intersection of $r_{k-2}$ and $I(r_k)$ it suffices to check that the left endpoint of $r_{k-2}$ has smaller $z$ coordinate than the left endpoint of $I(r_k)$. Using again \eqref{pointsXBx} and \eqref{pointsIxIbx} this leads to the formula

\begin{gather*}
Bxz(m,k-2) \leq Ixz(m,k) \Leftrightarrow \\ \Leftrightarrow -\frac{k}{\sqrt{1-(m-1)k(k-2)}} \leq -\frac{k+1}{\sqrt{(k+2)^2-m(k+1)(k+3)}} \Leftrightarrow \\ \Leftrightarrow \frac{2k^2+2k-1}{2k^2+2k} \geq m.
\end{gather*}
This completes the proof of the Lemma.
\end{proof}

As a final step towards the extension of results we need to present fixed points depending on $m$, such that their continuous period ratio is a $C^1$ function of $m$ with nonzero derivative. Luckily the fixed points identified in \eqref{fixedpoints} will do the job. Indeed, consider the $k$'th parameter interval from Lemma \ref{subsystemextended}, i.e. let $m$ be in $\bigl[\frac{1+k}{2+k}, \frac{2k^2+2k-1}{2k^2+2k}\bigr]$. Then the fixed points $F_{k-1}(m)$ and $F_k(m)$ are realised as physical configurations, they correspond to the projections $\pi(\underline{k-1})$ and $\pi(\underline{k})$. The ratio of their periods is

\begin{equation}\label{ratioextended}
\frac{\tau(F_k(m),m)}{\tau(F_{k-1}(m),m)} = \frac{(k+1)\sqrt{1-k^2 \alpha}}{k\sqrt{1-(k+1)^2 \alpha}},
\end{equation}
where $\alpha$ is as in \eqref{Fandalpha}.
Recall that in Theorem \ref{aediophantine} we stated rapid mixing for almost every value of the parameter within the interval $I \subset (1/2,1)$ from Theorem \ref{youngtower}, the existence of which was shown in our earlier paper \cite{BBNV}. We do not have a quantitative description of $I$, we only know that it is an open interval containing $0.74$ (though computer assisted techniques could give bounds on the endpoints of $I$). In any case, the proof of \cite{BBNV} that the discrete time system can be modelled by a Young tower with exponential tails works only for parameter values within $I$. On the other hand, as stated in Theorem \ref{complexity}, if we knew that the singularity set has subexponential complexity, then we would have a Young tower representation. Hence to extend our results we add the extra assumption of subexponential complexity (a usual assumption in the literature ensuring that expansion prevails cutting), which in particular implies the growth lemma and therefore the required tower representation.

\begin{proof}[Proof of Theorem \ref{assumingcomplexity}]
In view of our previous analysis the proof of Theorem \ref{aediophantine} can be repeated for any of the parameter intervals $\bigl[\frac{1+k}{2+k}, \frac{2k^2+2k-1}{2k^2+2k}\bigr]$ once we show that \eqref{ratioextended} is a $C^1$ function with nonzero derivative. Since $\alpha$ is a polynomial in $m$ and it is negative for $m \in (1/2,1)$ it follows that the denominator in \eqref{ratioextended} can not be zero and hence the ratio is indeed $C^1$ in $m$. A straightforward calculation shows that it is strictly decreasing for $m \in [1/2,3/4]$ and strictly increasing for $m \in [3/4,1]$. Therefore the proof of Theorem \ref{aediophantine} applies. To complete the proof we are left to verify that the parameter intervals $\bigl[\frac{1+k}{2+k}, \frac{2k^2+2k-1}{2k^2+2k}\bigr]$ cover the set $[2/3,1)$. We compare the right and left endpoints of consecutive intervals to check if they overlap.

\begin{equation*}
\frac{k+2}{k+3} \leq \frac{2k^2+2k-1}{2k^2+2k} \Leftrightarrow 0 \leq 2k^2+k-3
\end{equation*}
This relation holds for $k \geq 1$, therefore all the above parameter intervals overlap. Since both endpoints of the $k$'th interval converge to $1$ as $k$ goes to infinity, the statement of the theorem follows.
\end{proof}

\section*{Acknowledgements}
This research has been partially supported by OTKA (Hungarian National Fund for Scientific Research), grant K104745. The authors are grateful to Oliver Butterley, Ian Melbourne and Imre P\'eter T\'oth for valuable comments and enlightening discussions.

\section*{Appendix}

\setcounter{section}{1}
\setcounter{equation}{0}
\renewcommand{\thesection}{\Alph{section}}
\renewcommand{\theequation}{\Alph{section}.\arabic{equation}}

\begin{proof}[Proof of Lemma \ref{stableunstabledhc}]
It is enough to prove the part of the statement on the unstable direction, the stable part has the similar proof with $T$ replaced by $T^{-1}$. Let us first investigate the action of the tangent map $DT$ on unstable vectors. At every point the unstable cone is given by the union of the first and third quadrants of the plane (as shown in \cite{BBNV}). We first prove that the angle between two unstable vectors, lying in the same quadrant, is contracted by applying $DT$ on the vectors.
By \eqref{jacobian} it is clear that at every point $DT$ has the form $\bigl( \begin{smallmatrix} -1 & a \\ b & -1-ab \end{smallmatrix} \bigr)$ for some quantities $a, b < 0$. The image of a unit vector $(\cos \alpha, \sin \alpha)$ under the action of $DT$ is $(a \sin \alpha - \cos \alpha, b \cos \alpha - (1+ab)\sin \alpha)$. If $\alpha \in (0,\pi/2) \cup (\pi,3\pi/2)$, then the original vector is an unstable vector and so by uniform hyperbolicity there exists a constant $\Lambda > 1$ such that

\begin{equation}\label{expanded}
(a \sin \alpha - \cos \alpha)^2 + (b \cos \alpha - (1+ab)\sin \alpha)^2 \geq \Lambda > 1.
\end{equation}
The tangent map transforms the angles according to the function
\[ \alpha \to \arctan \Bigl(\frac{b\cos \alpha - (1+ab)\sin \alpha}{a\sin \alpha - \cos \alpha}\Bigr). \]
Differentiating this with respect to $\alpha$ many terms cancel out leaving
\[ \frac{1}{(a \sin \alpha - \cos \alpha)^2 + (b \cos \alpha - (1+ab)\sin \alpha)^2} \]
for the derivative, which is at most $1/\Lambda < 1$ for $\alpha \in (0,\pi/2) \cup (\pi,3\pi/2)$ using \eqref{expanded}. Therefore, by the mean value theorem we get

\begin{equation}\label{anglecontraction}
\angle(D_xT v_1, D_xT v_2) \leq \frac{1}{\Lambda}\angle(v_1, v_2),
\end{equation}
for all $x \in \mathcal{M}_1$ and $v_1, v_2 \in C_x^u$ with $\langle v_1, v_2\rangle > 0$. Since we can write the difference of two unit vectors $v_1$ and $v_2$ as $|v_1-v_2|^2 = 2(1-\cos \angle(v_1,v_2))$, using \eqref{anglecontraction} we get the formula

\begin{equation}\label{iterativecontraction}
\begin{split}
\biggl|\frac{D_xT v_1}{|D_xT v_1|}-\frac{D_xT v_2}{|D_xT v_2|}\biggr|^2 = \frac{1-\cos\angle(D_xT v_1,D_xT v_2)}{1-\cos\angle(v_1,v_2)} |v_1-v_2|^2 \leq \\ \leq \frac{1-\cos(\angle(v_1,v_2)/\Lambda)}{1-\cos\angle(v_1,v_2)} |v_1-v_2|^2 \leq \gamma^2 |v_1-v_2|^2,
\end{split}
\end{equation}
with some constant $\gamma \in (0,1)$, because $\frac{1-\cos(\alpha/\Lambda)}{1-\cos\alpha}$ is strictly less then $1$ for every $\alpha \in (0,\pi/2]$ and also in the limit as $\alpha \to 0^+$.
\\We now turn \eqref{iterativecontraction} into an iterative formula for the unstable direction. Given $u(x)$ and $u(y)$ we estimate the difference of the unstable directions in $T(x)$ and $T(y)$, assuming that $x$ and $y$ are in the same smoothness component of the dynamics. By the invariance of unstable manifolds we know that the unit vector tangent to the unstable manifold at $T(x)$ is nothing but $D_xT u(x)$ normalized. Hence we do the estimate

\begin{equation}\label{twotermunstable}
\biggl|\frac{D_xT u(x)}{|D_xT u(x)|}-\frac{D_yT u(y)}{|D_yT u(y)|}\biggr| \leq \biggl|\frac{D_xT u(x)}{|D_xT u(x)|}-\frac{D_yT u(x)}{|D_yT u(x)|}\biggr| + \biggl|\frac{D_yT u(x)}{|D_yT u(x)|}-\frac{D_yT u(y)}{|D_yT u(y)|}\biggr|.
\end{equation}
We remark that for this to be precise we identified the tangent spaces at $x$ and $y$. Also note that the unstable cone field for the system is constant, so after the identification both $u(x)$ and $u(y)$ can be viewed as unstable vectors in the same tangent space. Therefore, for the second term in \eqref{twotermunstable} we can apply \eqref{iterativecontraction} with $x$ replaced by $y$ and $v_1=u(x)$, $v_2=u(y)$.

For the first term we use our assumption of $x,y$ being in the same smoothness component and the assumption of Lemma \ref{stableunstabledhc} that this component is either $R_0$ or $R_1$. Within these circumstances we can use the piecewise $C^2$ property of the dynamics, in particular that there exists a constant $C_1>0$ such that
\[|D_xT v - D_yT v| \leq C_1 d(x,y),\]
for every unit vector $v$. Hence for the first term of \eqref{twotermunstable} we get

\begin{equation*}
\begin{split}
\biggl|\frac{D_xT u(x)}{|D_xT u(x)|}-\frac{D_yT u(x)}{|D_yT u(x)|}\biggr| & \leq \\ & \leq \frac{|D_xT u(x)-D_yT u(x)|}{|D_xT u(x)|} + \frac{||D_yT u(x)|-|D_xT u(x)||}{|D_xT u(x)|} \leq \\ & \leq \frac{2|D_xT u(x)-D_yT u(x)|}{|D_xT u(x)|} \leq \frac{2}{\Lambda} C_1 d(x,y).
\end{split}
\end{equation*}
This, together with our previous observation on the second term leads to

\begin{equation*}
\Bigl|\frac{D_xT u(x)}{|D_xT u(x)|}-\frac{D_yT u(y)}{|D_yT u(y)|}\Bigr| \leq \frac{2}{\Lambda} C_1 d(x,y) + \gamma |u(x)-u(y)|,
\end{equation*}
which can be iterated as long as the image of $x$ and $y$ are in the same smoothness component, i.e. as long as they are not separated by a singularity. This iteration results in

\begin{equation}\label{iteratedunstable}
\biggl|\frac{D_xT^n u(x)}{|D_xT^n u(x)|}-\frac{D_yT^n u(y)}{|D_yT^n u(y)|}\biggr| \leq \sum\limits_{i=0}^{n-1} \frac{2C_1}{\Lambda} \gamma^i d(T^i(x),T^i(y)) + \gamma^n |u(x)-u(y)|,
\end{equation}
for any $0 < n < s_+(x,y)$. To get an estimate on $|u(x)-u(y)|$ (as required in Lemma \eqref{stableunstabledhc}) we replace in \eqref{iteratedunstable} $x$ and $y$ by $T^{-n}(x)$ and $T^{-n}(y)$ respectively and get

\begin{equation}\label{finalunstable}
|u(x)-u(y)| \leq \sum\limits_{i=0}^{n-1} \frac{2C_1}{\Lambda} \gamma^i d(T^{i-n}(x),T^{i-n}(y)) + \gamma^n |u(T^{-n}(x))-u(T^{-n}(y))|.
\end{equation}
Now observe that
\begin{gather*}
s_+(T^{i-n}(x),T^{i-n}(y))=s_+(x,y)+n-i \\
s_-(T^{i-n}(x),T^{i-n}(y))=s_-(x,y)+i-n,
\end{gather*}
as long as $0 \leq n \leq s_-(x,y)$. Hence for all $i \in \{0, \dots, n-1\}$ we have
\[ s(T^{i-n}(x),T^{i-n}(y)) \geq \min\{s_+(x,y)+1,s_-(x,y)-n\}, \]
so choosing $n=s_-(x,y)/2$ the inequality
\[ s(T^{i-n}(x),T^{i-n}(y)) \geq \frac12 s(x,y) \]
holds. Therefore by \eqref{metricsequiv} still for every $i \in \{0, \dots, n-1\}$ we have
\[ d(T^{i-n}(x),T^{i-n}(y)) \leq C \theta^{\frac{s(x,y)}{2}} \]
and since $n \geq s(x,y)/2$ by our choice, we get the final estimate for \eqref{finalunstable}

\begin{equation*}
|u(x)-u(y)| \leq \frac{2C_1}{\Lambda} \frac{1}{1-\gamma} C \theta^{\frac{s(x,y)}{2}} + 2 \gamma^{\frac{s(x,y)}{2}},
\end{equation*}
proving the lemma with $\gamma_u = \max\{\sqrt{\theta},\sqrt{\gamma}\}$.
\end{proof}

\bibliographystyle{plain}

\end{document}